\documentclass[11pt,a4paper]{amsart}

\usepackage{a4wide}
\usepackage[sort,noadjust,space,compress]{cite}
\usepackage{eucal}
\usepackage{amsmath}
\usepackage{amssymb}
\usepackage{mathtools}
\usepackage{array}
\usepackage{enumitem}
\setlist{leftmargin=*}
\usepackage{caption}

\usepackage{microtype}

\usepackage{tikz} 
\usepackage{tikz-cd}
\usepackage{tkz-graph}
\usepackage{tkz-euclide}
\usepackage{tkz-berge}
\tikzset{MyLabel/.style={
   auto=right,
   fill=none,
   outer sep=0.2ex}}

\theoremstyle{plain}
\newtheorem{theorem}{Theorem}[section]
\newtheorem{proposition}[theorem]{Proposition}
\newtheorem{lemma}[theorem]{Lemma}

\theoremstyle{definition}
\newtheorem{definition}[theorem]{Definition}
\theoremstyle{remark}

\newtheorem{construction}[theorem]{Construction}

\newcommand{\bbT}{\mathbb{T}}

\newcommand{\injto}{\hookrightarrow}
\newcommand{\Aut}{\operatorname{Aut}}

\newcommand{\restr}{\mathord{\upharpoonright}}

\newcommand{\dotcup}{\mathrel{\dot\cup}}
\newcommand{\epito}{\twoheadrightarrow}
\newcommand{\Cl}{\operatorname{Cl}}
\newcommand{\bb}{\theta}
\newcommand{\GQ}{\operatorname{GQ}}
\newcommand{\PQ}{\operatorname{PQ}}

\newcommand{\lseq}{\prescript{}{=}{}}

\newcommand{\icol}[1]{% inline column vector
  \left(\substack{#1}\right)%
}

{\end{tikzpicture}}

{\end{tikzpicture}}

\newcommand{\GL}{\operatorname{GL}}
\newcommand{\tz}{\bar{z}}
\newcommand{\ta}{\bar{a}}
\newcommand{\tb}{\bar{b}}
\newcommand{\tc}{\bar{c}}

\newcommand{\tv}{\bar{v}}
\newcommand{\tu}{\bar{u}}
\newcommand{\tw}{\bar{w}}
\newcommand{\tn}{\bar{0}}
\newcommand{\te}{\bar{e}}
\newcommand{\tx}{\bar{x}}
\newcommand{\ty}{\bar{y}}
\newcommand{\F}{\mathbb{F}}
\newcommand{\type}{\mathbb}

\newcommand{\hq}{\widehat{q}}
\newcommand{\hS}{\widehat{S}}
\newcommand{\hQ}{\widehat{Q}}

\newcommand{\hGamma}{\widehat{\Gamma}}
\newcommand{\hcC}{\widehat{\mathcal{C}}}

\author[Ch.\,Pech]{Christian Pech}
\address{Radebeul, Germany}
\email{cpech@freenet.de}
\urladdr{https://www.researchgate.net/profile/Christian\_Pech2}

\author[M.\,Pech]{Maja Pech} 
\address{University of Novi Sad\\Faculty of Sciences\\Department of Mathematics and Informatics} 
\email{maja@dmi.uns.ac.rs}
\urladdr{http://people.dmi.uns.ac.rs/~maja/}

\title{On a family of highly regular graphs by Brouwer, Ivanov, and Klin}

\begin{document}

\begin{abstract}
	Highly regular graphs for which not all regularities are explainable by symmetries are fascinating creatures. Some of them like, e.g., the line graph of W.~Kantor's non-classical $\GQ(5^2,5)$, are  stumbling stones for existing implementations of graph isomorphism tests. They appear to be extremely rare and even once constructed it is difficult to prove their high regularity. Yet some of them, like the McLaughlin graph on 275 vertices and Ivanov's graph on 256 vertices are of profound beauty.  This alone makes it an attractive goal to strive for their complete classification or, failing this, at least to get a deep understanding of them. Recently, one of the authors discovered new methods for proving high regularity of graphs. Using these techniques, in this paper we study a classical family of strongly regular graphs, originally discovered by A.E.~Brouwer, A.V.~Ivanov, and M.H.~Klin in the late 80s. We analyze their symmetries and show that they are $(3,5)$-regular but not $2$-homogeneous. Thus we promote these graphs to the distinguished club of highly regular graphs with few symmetries.
\end{abstract}
\maketitle

\section{Introduction}

Recall that a simple graph $\Gamma$ is called \emph{regular} if there exists a number $k$, such that each vertex of $\Gamma$ has exactly $k$ neighbors. The concept of regularity can be extended naturally. Roughly speaking for a given configuration of vertices in $\Gamma$ we may count extensions of this configuration to a bigger, given, type of configuration. An example is given by the $k$-isoregular graphs.  A regular graph is called $k$-isoregular graph if for every induced subgraph $\Delta\le\Gamma$ the number of joint neighbors of $V(\Delta)$ in $\Gamma$ depends only on the isomorphism type of $\Delta$. When we talk about high regularity, we have in mind a much more general set of regularity conditions:
\begin{definition}
	A \emph{graph type} $\type{T}$ of order $(m,n)$ is a triple $(\Delta,\iota,\Theta)$, where $\Delta$ and $\Theta$ are graphs of order $m$ and $n$, respectively, and where $\iota\colon\Delta\injto\Theta$ is an embedding. A graph $\Gamma$ is called \emph{$\type{T}$-regular} if either $\Delta$ does not embed into $\Gamma$ or if for all $\kappa\colon\Delta\injto\Gamma$ the number $\#(\Gamma,\type{T},\kappa)$ of embeddings $\hat\kappa\colon\Theta\injto\Gamma$ with $\kappa=\hat\kappa\circ\iota$ does not depend on $\kappa$ (i.e., it is equal to  a constant $\#(\Gamma,\type{T})$).  %(i.e., it does not depend on $\kappa$). 
	 In the case that $\Delta$ does not embed into $\Gamma$ we define  $\#(\Gamma,\type{T})$ to be equal to $0$.
\end{definition}
We are usually not so much interested into regularities for particular graph types but rather for whole classes. 
\begin{definition}
	A graph is called
	\begin{itemize}
		\item \emph{$(\lseq m, \lseq n)$-regular}, if it is $\type{T}$-regular for each graph type $\type{T}$ of order $(m,n)$,
        \item \emph{$(\lseq m,  n)$-regular}, if it is $(\lseq m, \lseq l)$-regular for all $m\le l\le n$,
        %\item $( m, \lseq n)$-regular if it is $(\lseq k, \lseq n)$-regular for all  $k\le m$,
        \item \emph{$( m, n)$-regular}, if it is $(\lseq k,  n)$-regular for all  $k\le m$.
	\end{itemize}
\end{definition}
The notion of $(m,n)$-regularity generalizes several classical regularity-concepts for graphs. E.g., the $( 2, 3)$-regular graphs coincide with the \emph{strongly regular graphs} (in the sense of Bose \cite{Bos63}), the $(2,t)$-regular graphs correspond to the  graphs that satisfy the \emph{$t$-vertex condition} (in the sense of Higman \cite{Hig71}, cf.~also \cite{HesHig71}). Finally, the $(k,k+1)$-regular graphs coincide with the \emph{$k$-regular graphs} (in the sense of Gol'fand and Klin \cite{GolKli78}) and with the  \emph{$k$-tuple regular graphs} (in the sense of Buczak \cite{Buc80}). Nowadays, in order to avoid conflicts with existing graph-theoretical terminology, these graphs are called \emph{$k$-isoregular} (cf.~\cite{KliPoeRos88}).
\begin{definition}
	We call a graph $\Gamma$ \emph{highly regular} if there is some $m\ge 2$ and some $n\ge 4$, such that $\Gamma$ is  $( m, n)$-regular.
\end{definition}  
Note that our definition of high regularity excludes the strongly regular graphs that do not satisfy the $4$-vertex condition. The reason for this is that we are ultimately interested in a  classification of highly regular graphs. However, such a classification for strongly regular graphs in general seems hopeless as for certain orders there are so-called prolific constructions (cf.~\cite{Wal71,FDF02,Muz07}).

Most naturally, regularity is induced by symmetry. E.g., if a graph is vertex transitive, then it is also regular. Recall that a graph is called \emph{symmetric} if its automorphism group acts transitively on vertices and arcs (cf. \cite{GodRoy01}). When we talk about highly symmetric graphs, we think about  even stronger conditions:
\begin{definition}
	Let $\Gamma$ and $\Delta$ be graphs. Then $\Gamma$ is called \emph{$\Delta$-homogeneous} if for all $\iota_1,\iota_2\colon\Delta\injto\Gamma$ there exists $\alpha\in\Aut(\Gamma)$ such that $\iota_2=\alpha\circ\iota_1$. 
	It is called \emph{weakly $\Delta$-homogeneous} if for all $\iota_1,\iota_2\colon\Delta\injto\Gamma$ there exist $\alpha\in\Aut(\Gamma)$ and $\beta\in\Aut(\Delta)$, such that $\alpha\circ\iota_1 = \iota_2\circ\beta$.
\end{definition}
Note that many of the common symmetry-conditions naturally translate into special cases of this definition. For instance, vertex transitivity is $K_1$-homogeneity, arc-transitivity is $K_2$-homogeneity, edge-transitivity is weak $K_2$-homogeneity\dots.  
\begin{definition}
	A graph $\Gamma$ is called \emph{$k$-homogeneous} if it is $\Delta$-homogeneous for all graphs $\Delta$ of order $\le k$.  It is called \emph{homogeneous} if it is $k$-homogeneous, for every $k>0$.
\end{definition}
In general, we call a graph  \emph{highly symmetric} if it is $k$-homogeneous, for some $k\ge 2$.  High symmetry implies high regularity:  It is easy to see that every $k$-homogeneous graph is $(k,l)$-regular, for every $l\ge k$. 
Note that the highly symmetric graphs are completely classified up to isomorphism. The homogeneous finite graphs were classified by Gardiner,  Gol'fand and Klin \cite{Gar76,GolKli78}. It was shown by Cameron \cite{Cam80} that every $5$-homogeneous graph is homogeneous. The finite $4$-homogeneous graphs were characterized by Buczak \cite{Buc80}. It turns out that there is up to isomorphism and up to complement just one $4$-homogeneous graph that is not homogeneous, the Schl\"afli graph. The $3$-homogeneous graphs were classified by Cameron and Macpherson \cite{CamMac85}. Finally, the $2$-homogeneous graphs are implicitly known by the classification of rank-3-groups that was carried out by  Bannai, Kantor, Liebler, Liebeck, and Saxl (\cite{Ban72,KanLie82,Lie87,LieSax86}). It is noteworthy that the classification of the $k$-homogeneous graphs for $2\le k\le 4$ relies on the classification of finite simple groups.   

We are mostly interested in highly regular graphs for which not all regularities are explainable by symmetries. Apart from the sheer intellectual  challenge to classify these combinatorial objects, we are interested in such graphs since they play a role in the research about the complexity of the graph isomorphism problem. For existing implementations of graph isomorphism tests (like, e.g., the widely used package \emph{nauty} by B.~McKay \cite{McKPip14}) highly regular graphs with few symmetries form a performance bottleneck. For instance, in its standard settings it takes hours of cpu-time for nauty to compute a canonical labeling of the line graph of the $\GQ(5^2,5)$ constructed by Kantor in \cite{Kan80} (cf.~also \cite{Pay92,PayTha09}). Here the notion $\GQ(s,t)$ refers to \emph{generalized quadrangles} of order $(s,t)$ in the sense of Tits \cite{Tit59}.

Interestingly, there exist highly symmetric graphs for which not all regularities are explainable by symmetries. E.g., the McLaughlin graph on $275$ vertices is $( 4, 5)$-regular but is not $4$-homogeneous. So in  particular it is highly regular. On the other hand it is $3$-homogeneous and thus, according to our definition, it is highly symmetric.

While we know almost everything about highly symmetric graphs, our knowledge about highly regular graphs is still very modest. This is so, even though a considerable amount of research went into their classification during the last few decades. It is generally not so hard to construct a graph with given regularities, but it is much harder to construct one with few symmetries.  The following timeline shows roughly the development of the research since the early seventies:
\begin{description}
	\item[1970] Hestenes and Higman introduce the notion of $(2,t)$-regularity and show that point graphs of generalized quadrangles are $(2,4)$-regular (cf. \cite{Hig71,HesHig71}.
    \item[1984] Farad\v{z}ev, A.A.Ivanov, and Klin  construct a $(2,3)$-regular graph on 280 vertices with Aut($J_2$) as automorphism group, that is not $2$-homogeneous (cf.~\cite{IvaKliFar84,FarKliMuz94}).
    \item[1989] A.V.Ivanov  finds a $(2,5)$-regular graph on 256 vertices, whose subconstituents, both, are $(2,4)$-regular (cf.~\cite{Iva89}).
    \item[1989] Brouwer, Ivanov, and Klin describe a family $\Gamma^{(m)}$ of $(3,4)$-regular graphs that contains Ivanov's graph  as $\Gamma^{(4)}$, and show that their first subconstituents are $(2,4)$-regular but not $2$-homogeneous (cf.~\cite{BroIvaKli89}).
    \item[1994] A.V.Ivanov  discovers another infinite family $\widehat{\Gamma}^{(m)}$ of $(2,4)$-regular graphs (cf.~\cite{Iva94}).
    \item[2000] Reichard shows that both, $\Gamma^{(m)}$ and $\widehat{\Gamma}^{(m)}$ are  $(2,5)$-regular, for all $m\ge 4$. Moreover, he shows that the graph  discovered in 1984 is $(2,4)$-regular (cf.~\cite{Reich00}).
    \item[2003] Reichard shows that point graphs of $\GQ(s,t)$ are $(2,5)$-regular, and that the point graphs of $\GQ(q,q^2)$ are $(2,6)$-regular (cf.~\cite{Reich03}). 
    \item[2003] Klin, Meszka, Reichard, and Rosa identify the smallest $(2,4)$-regular graph that is not $2$-homogeneous. It has parameters $(v,k,\lambda,\mu)=(36,14,4,6)$ (cf.~\cite{KliMesReiRos05}).
    \item[2004] CP shows that the  point graphs of $\PQ(s,t,\mu)$ are $(2,5)$-regular. Here $\PQ(s,t,\mu)$ refers to the \emph{partial quadrangles} of order $(s,t,\mu)$ in the sense of Cameron \cite{Cam75}.
    \item[2005] Reichard shows that the point graphs of $\GQ(q,q^2)$ are $(2,7)$-regular (cf.~\cite{Rei15}).
    \item[2007] CP shows that the point graphs of $\PQ(q-1,q^2,q^2-q)$ are $(2,6)$-regular.
    \item[2007] Klin and CP find two self-complementary $(2,4)$-regular graphs that are not $2$-homoge\-neous.
    \item[2014] CP shows that the point graphs of $\GQ(q,q^2)$ are $(3,7)$-regular (cf.~\cite{Pec14}).
\end{description}
As can be read of this timeline much work had to be put into proving high regularity for graphs that were already known. The reason for the difficulties is that with growing $m$ and $n$ we experience a combinatorial explosion of the number of graph types of order $(m,n)$. For instance, there are 20.364 pairwise non-isomorphic graph types of order $(3,7)$.

Interestingly, sometimes high regularity implies high symmetry. It was shown independently by Buczak, Gol'fand, and Cameron, that every $(5,6)$-regular graph is already homogeneous. Thus, when classifying highly regular graphs, we may restrict our attention to $(m,n)$-regular graphs for which $m<n$ and for which $m<5$. Indeed, as was mentioned above, there is known only one $(4,5)$-regular graph that is not $4$-homogeneous---the McLaughlin graph. Non-$3$-homogeneous, $(3,t)$-regular graphs for $t\ge 4$ appear to be extremely rare. In this paper we are going to uncover another infinite family of $(3,5)$-regular graphs that are not $2$-homogeneous. We will do so by giving a complete analysis of the family $\Gamma^{(m)}$ originally discovered by Brouwer, Ivanov, and Klin. 

\section{Constructions and results}
In \cite{Iva89} A.V.~Ivanov constructed  a  $( 2, 5)$-regular graph $\Gamma^{(4)}$ with $256$ vertices and valency $120$.  The automorphism group of this graph acts transitively on vertices and arcs but not on non-arcs. In particular, $\Gamma^{(4)}$ is not $2$-homogeneous. 
Ivanov showed further that the first and second subconstituents of $\Gamma^{(4)}$ are $( 2, 4)$-regular. Here the \emph{first (the second) subconstituent} of a graph $\Gamma$ with respect to a vertex $v\in V(\Gamma)$ is the subgraph of $\Gamma$ induced by all the neighbors (all the non-neighbors) of $v$ in $\Gamma$. The first and the second subconstituent of $\Gamma$ with respect to $v$ are denoted by $\Gamma_1(v)$ and by $\Gamma_2(v)$, respectively. Clearly, if $\Aut(\Gamma)$ acts transitively on vertices, then all first subconstituents (all second subconstituents) are mutually isomorphic. In this case, if the vertex with respect to which we take the subconstituent is not important, then instead of $\Gamma_i(v)$ we write just $\Gamma_i$ ($i\in\{1,2\}$).   

It is well known that a strongly regular graph $\Gamma$ is $( 3, 4)$-regular if and only if its  subconstituents $\Gamma_i(v)$ are strongly regular with parameters independent from $v\in V(\Gamma)$ (for a proof see, e.g., \cite[Proposition 4]{Rei15}). Thus, Ivanov's graph $\Gamma^{(4)}$ is $( 3, 4)$-regular.  

In \cite{BroIvaKli89} a wide class of strongly regular graphs is described of which Ivanov's graph is a special case. We are not going to repeat the construction in full generality but only  as far as it touches our interests. In particular, only one  series of strongly regular graphs from  \cite{BroIvaKli89} consists of $( 3, 4)$-regular graphs.  This is the one that we consider in the sequel. A first construction goes as follows:  
\begin{construction}[{\cite{BroIvaKli89}}]\label{const1}
	Consider the vector space $\F_2^{2m}$. Let $q\colon \F_2^{2m}\to\F_2$ be a non-degenerate quadratic form over $\F_2^{2m}$ of maximal Witt index. Let $Q\subseteq\F_2^{2m}$ be the quadric defined by $q$, and let $S\le \F_2^{2m}$ be a maximal singular subspace for $q$. Now define  $\Gamma^{(m)}\coloneqq (V^{(m)},E^{(m)})$ according to
	\begin{equation*}
		V^{(m)} \coloneqq  \F_2^{2m},\quad E^{(m)} \coloneqq  \{(\tv,\tw)\mid \tw-\tv\in Q\setminus S\}.
	\end{equation*}
\end{construction}
A first analysis of these graphs was given in \cite{BroIvaKli89}. Further steps were taken in  \cite{Iva94} and \cite{Reich00}. In the following we collect what is known about the graphs $\Gamma^{(m)}$ and their subconstituents and what is relevant for this paper:
\begin{itemize}
	\item $\Gamma^{(m)}$ is strongly regular (\cite[Section 2]{BroIvaKli89}),
	\item $\Gamma^{(m)}$ is symmetric (\cite[Section 3]{Iva94}),
	\item  the first subconstituent $\Gamma_1^{(m)}$ is $( 2, 4)$-regular (\cite[Theorem 1]{BroIvaKli89}),
	\item the second subconstituent $\Gamma_2^{(m)}$ is strongly regular (\cite[Section 4]{BroIvaKli89}),
	\item  if  $m\ge 4$ then $\Aut(\Gamma_1^{(m)})$ has rank $4$ (\cite[Theorem 1]{BroIvaKli89}), 
	\item $\Gamma^{(m)}$ is $( 3, 4)$-regular (this is a direct consequence of the previous items (cf.\ also \cite[Proposition 4]{Rei15}),
	\item $\Gamma^{(m)}$ is $( 2, 5)$-regular (\cite[Theorem 7]{Reich00}).
	%\item both $\Gamma^{(m)}$ and $\Gamma_1^{(m)}$ are in the switching class of a regular two-graph,
	%\item $\Gamma_2^{(m)}$ is obtained from a regular two-graph by isolating a vertex.
\end{itemize} 
The parameters of $\Gamma^{(m)}$ and its subconstituents are given in the following table. Here and below, in order to save space and to improve readability, we denote the number $2^{m-3}$ by $\theta_m$.
\[
\resizebox{\textwidth}{!}{$\displaystyle
\renewcommand{\arraystretch}{2.0}\begin{array}{l|cccccccc}
 & \multicolumn{1}{c}{v} & \multicolumn{1}{c}{k} & \multicolumn{1}{c}{\lambda} & \multicolumn{1}{c}{\mu} & \multicolumn{1}{c}{r} & \multicolumn{1}{c}{s} & \multicolumn{1}{c}{f} & \multicolumn{1}{c}{g} \\\hline
 \Gamma^{(m)} & 64\bb_m^2 & 4\bb_m(8\bb_m-1) & 4\bb_m(4\bb_m-1) & 4\bb_m(4\bb_m-1) & 4\bb_m & -4\bb_m & 4\bb_m(8\bb_m-1) & (4\bb_m+1)(8\bb_m-1)	\\
 \Gamma_1^{(m)} & 4\bb_m(8\bb_m-1) & 4\bb_m(4\bb_m-1) & 2\bb_m(4\bb_m-1) & 4\bb_m(2\bb_m-1) & 4\bb_m & -2\bb_m & \dfrac{(4\bb_m-1)(8\bb_m-1)}{3} & \dfrac{4(4\bb_m-1)(4\bb_m+1)}{3}\\
 \Gamma_2^{(m)} & (4\bb_m+1)(8\bb_m-1) & 16\bb_m^2 & 2\bb_m(4\bb_m-1) & 8\bb_m^2 & 2\bb_m & -4\bb_m & \dfrac{4(4\bb_m+1)(4\bb_m-1)}{3} & \dfrac{2(2\bb_m+1)(8\bb_m-1)}{3}
\end{array}$}
\]
Let us make Construction~\ref{const1} more concrete:
\begin{construction}
	Construction~\ref{const1} requires a non-degenerate quadratic form on $\F_2^{2m}$ of maximal Witt index.  Up to equivalence, there is exactly one such quadratic form and its Witt index is $m$. Moreover, it does not matter which quadratic form from this equivalence class we choose as two equivalent forms will lead to isomorphic graphs. For the rest of the paper we will consider  
	\[
	q^{(m)}(x_1,\dots,x_m,y_1,\dots,y_m)\coloneqq \sum_{i=1}^m x_iy_i.
	\]
	It is convenient to identify $\F_2^{2m}$ with the isomorphic vector space $(\F_2^m)^2$ whose elements are of the shape $\tv=\icol{\tv_1\\\tv_2}$, where $\tv_1$ and $\tv_2$ are binary vectors of length $m$. Note that with this identification we have 
	\[
	q^{(m)}(\tv)= q^{(m)}\icol{\tv_1\\\tv_2}= \tv_1^{T}\tv_2.
	\]
	The quadric $Q_m$ induced by $q^{(m)}$ consists of all vectors $\tv$ such that $\tv_1^T\tv_2=0$. A maximal singular subspace $S_m$  is given by the set of all $\tv\in\F_2^{2m}$ for which $\tv_2=\tn$ (here and below, by $\tn$ we denote the zero vector; in each case the length of $\tn$ will be clear from the context). Now we can repeat the construction of $\Gamma^{(m)}=(V^{(m)},E^{(m)})$ in more concrete terms:
	\[
		V^{(m)}=\F_2^{2m},\quad E^{(m)}=\{(\tv,\tw)\mid (\tv_1+\tw_1)^T(\tv_2+\tw_2)=0, \tv_2\neq\tw_2\}. 
	\]
	From now on, whenever we talk about the graphs $\Gamma^{(m)}$, we have in mind this model.
\end{construction}

Our main result is:
\begin{theorem}\label{mainthm}
	Let $m\ge 4$ be a natural number. Then
	\begin{enumerate}
		\item $\Gamma^{(m)}$ is not $2$-homogeneous; the orbitals of $\Aut(\Gamma^{(m)})$ are given  by the following binary relations on $\F_2^{2m}$:
		\begin{align*}
		\varrho_1^{(m)} &= \{(\tv,\tw)\mid \tv=\tw\}, & \varrho_2^{(m)} &= \{(\tv,\tw)\mid \tv+\tw\in S_m\setminus\{\tn\}\,\},\\
		\varrho_3^{(m)} &= \{ (\tv,\tw)\mid  \tv+\tw\in Q_m\setminus S_m\}, & \varrho_4^{(m)} &= \{(\tv,\tw)\mid \tv+\tw\notin Q_m\};
		\end{align*}
		\item the relational structure $\mathcal{C}^{(m)}\coloneqq (\F_2^{2m};\,\varrho_1^{(m)},\varrho_2^{(m)},\varrho_3^{(m)},\varrho_4^{(m)})$ is $3$-homogeneous, i.e., every isomorphism between relational substructures of at most three elements extends to an automorphism,
		\item $\Gamma^{(m)}$ is $( 3, 5)$-regular,
		\item $\Gamma_1^{(m)}$ is $(2,4)$-regular but not $2$-homogeneous (already known from \cite{BroIvaKli89}),
		\item $\Gamma_2^{(m)}$ is $(2,4)$-regular but not $1$-homogeneous.
	\end{enumerate}
\end{theorem}
The rest of the paper is devoted to the proof of this result. 

\section{Symmetries of the graphs $\Gamma^{(m)}$}
Let us  have a look onto the automorphisms of $\Gamma^{(m)}$. By $\GL(n,2)$ we denote the group of regular $n\times n$-matrices over $\F_2$. Clearly, $\Gamma^{(m)}$ is invariant under all affine transformations $\varphi_{A,\tw}\colon\tv\mapsto A\tv+\tw$ for which $A\in\GL(2m,2)$ preserves $Q_m$ and $S_m$ setwise. Let us denote this group  by $G_m$ and the stabilizer of $\tn$ in $G_m$ by $H_m$. Then we have 
\begin{lemma}\label{HM}
	Let $A\in \GL(m,2)$, and let $S$ be a symmetric $m\times m$-matrix over $\F_2$ with $0$ diagonal. Then 
	\[
	\begin{pmatrix}
		A & AS\\
		O & (A^T)^{-1}
	\end{pmatrix}
	\]
	is an element of $H_m$ (here and below $O$ denotes the zero-matrix). Moreover, every element of $H_m$ is obtained in this way. 
\end{lemma}
\begin{proof}
	Let $M\in H_m$. Then $M$ can be decomposed into $m\times m$-blocks like
	\[
	M=\begin{pmatrix}
		A & B \\
		D & C
	\end{pmatrix}.
	\]
	As $M$ preserves $S_m$, we have $D=O$. Let $\tv\in V(\Gamma^{(m)})$. Then
	\[
	 M\tv = \begin{pmatrix}
	 	A & B \\
	 	O & C
	 \end{pmatrix}
	 \begin{pmatrix}
	 	\tv_1\\
	 	\tv_2
	 \end{pmatrix}= \begin{pmatrix}
	 	A\tv_1+B\tv_2\\ C\tv_2
	 \end{pmatrix}.
	\]
	 Since $M$ preserves $Q_m$, we have for all  $\tv\in Q_m$ that $(A\tv_1+B\tv_2)^T C\tv_2=0$. That means
	 \[
	 	0 = (A\tv_1)^TC\tv_2 + (B\tv_2)^TC\tv_2 = \tv_1^T A^T C\tv_2 + \tv_2^T B^T C\tv_2.
	 \]
	 If we consider the special case that $\tv_1=\tn$, then we obtain $\tv_2^T B^T C \tv_2= 0$, for all $\tv_2\in \F_2^m$. From this it follows that $S\coloneqq B^T C$ is a symmetric matrix with $0$-diagonal. Moreover, it follows that  $\tv_1^T A^T C\tv_2=0$, for all $\tv\in Q_m$. However, from this it follows that $A^T C=I$. Indeed, if we consider all vectors of the shape $\tv=(\te_i,\te_j)^T$ (here and below, by $\te_i$ we denote the vector whose $i$-th entry is equal to $1$ and whose remaining entries are equal to $0$; in each case the length of $\te_i$ will be clear from the context) for $i\neq j$, then we obtain that all the off-diagonal entries of $A^T C$ are equal to $0$. Since both, $A$ and $C$ are regular, the claim follows. Now we may conclude that $C= (A^T)^{-1}$ and $S= B^T(A^T)^{-1}= (A^{-1}B)^T = A^{-1}B$. It follows that $B=AS$. Thus we showed that every element of $H_m$ is of the desired shape. It is not hard to see that every matrix of this shape preserves $Q_m$ and $S_m$ setwise. 
\end{proof}
%\begin{remark}
%	Since $\Aut(\Gamma^{(m)})$ is transitive, it  follows that all second subconstituents of $\Gamma^{(m)}$ are pairwise isomorphic. Henceforth, by $\Gamma_2^{(m)}$ we denote any second subconstituent of $\Gamma^{(m)}$. 
%\end{remark}
\begin{proposition}\label{notrankthree}
	For all $m\ge 4$ the graph $\Gamma^{(m)}$ is not $2$-homogeneous, and $\Gamma_2^{(m)}$ has an intransitive automorphism group. 
\end{proposition}
\begin{proof}
	Consider the vertices $\tz=\icol{\tn\\\tn}$, $\ta=\icol{\te_1\\\tn}$, $\tb=\icol{\te_1\\\te_1}$. Clearly, $(\tz,\ta)$ and $(\tz,\tb)$ are non-arcs in $\Gamma^{(m)}$. Our goal is to show that no automorphism of $\Gamma^{(m)}$ maps $(\tz,\ta)$ to $(\tz,\tb)$: To this end we introduce two auxiliary graphs $\Upsilon_{\ta}$ and $\Upsilon_{\tb}$. The vertices of $\Upsilon_{\ta}$ and $\Upsilon_{\tb}$ shall be all neighbours of $\tz$ in $\Gamma^{(m)}$ that are non-neighbours of $\ta$ and $\tb$, respectively. If we can show that $\Upsilon_{\ta}$ and $\Upsilon_{\tb}$ are non-isomorphic, then we are done.
	
	It is not hard to see that we have
	\begin{align*}
		V(\Upsilon_{\ta}) &= \{\icol{\tv_1\\\tv_2}\mid \tv_1^T\tv_2=0, \tv_2(1)=1\},\\
		V(\Upsilon_{\tb}) &= W_1\dotcup W_2,\text{ where} \\	
		W_1 &= \{\icol{\tv_1\\\tv_2}\mid \tv_2=\te_1,\tv_1(1)=0\},\\
		W_2 &= 	\{\icol{\tv_1\\\tv_2}\mid \tv_2\neq\tn, \tv_1^T\tv_2=0,\tv_1(1)=\tv_2(1)\},
	\end{align*}
	where for a vector $\tv$, by $\tv(i)$ we denote the $i$-th entry of $\tv$.	
	In order to understand the structure of the graphs $\Upsilon_{\ta}$ and $\Upsilon_{\tb}$, consider the projection $\Pi\colon\F_2^{2m}\epito\F_2^{2m-2}$ given by
	\[
	\Pi\colon \F_2^{2m}\epito\F_2^{2m-2}: \icol{x_1\\\vdots\\x_m\\y_1\\\vdots\\y_m}\mapsto \icol{x_2\\\vdots\\x_m\\y_2\\\vdots\\y_m}.
	\]
	Note that the restrictions of $\Pi$ to $V(\Upsilon_{\ta})$ and to $V(\Upsilon_{\tb})$ both are bijections with $\F_2^{2m-2}$. Moreover, if we define
	\begin{align*}
		A_1 &\coloneqq \left\{\icol{0\\\tn\\1\\\tn}\right\}, & B_1 &\coloneqq \left\{\icol{0\\\tn\\1\\\tn}\right\},\\
		A_2 &\coloneqq \left\{\icol{0\\\tx\\1\\\tn}\mid \tx\in\F_2^{m-1}\setminus\{\tn\}\right\}, & B_2 &\coloneqq \left\{\icol{0\\\tx\\1\\\tn}\mid \tx\in\F_2^{m-1}\setminus\{\tn\}\right\},\\
		A_3 &\coloneqq  \left\{\icol{0\\\tx\\1\\\ty}\mid \tx,\ty\in\F_2^{m-1},\tx^T\ty=0,\ty\neq\tn\right\}, & B_3 &\coloneqq  \left\{\icol{0\\\tx\\0\\\ty}\mid \tx,\ty\in\F_2^{m-1},\tx^T\ty=0,\ty\neq\tn\right\},\\
		A_4 &\coloneqq  \left\{\icol{1\\\tx\\1\\\ty}\mid \tx,\ty\in\F_2^{m-1},\tx^T\ty=1\right\}, & B_4 &\coloneqq  \left\{\icol{1\\\tx\\1\\\ty}\mid \tx,\ty\in\F_2^{m-1},\tx^T\ty=1\right\}.
	\end{align*}
	Then $V(\Upsilon_{\ta})=A_1\dotcup A_2\dotcup A_3\dotcup A_4$, $V(\Upsilon_{\tb})=B_1\dotcup B_2\dotcup B_3\dotcup B_4$, and 
	\begin{align*}
		\Pi(A_1) &=\Pi(B_1) = \{\tn\}, & 
		\Pi(A_2) &=\Pi(B_2) = S_{m-1}\setminus\{\tn\},\\
		\Pi(A_3) &=\Pi(B_3) = Q_{m-1}\setminus S_{m-1}, & 
		\Pi(A_4) &=\Pi(B_4) = \F_2^{2m-2}\setminus Q_{m-1}.
	\end{align*}
	The edges of $\Upsilon_{\ta}$ and $\Upsilon_{\tb}$ may be read off the following diagrams:
	\[\scalebox{0.9}
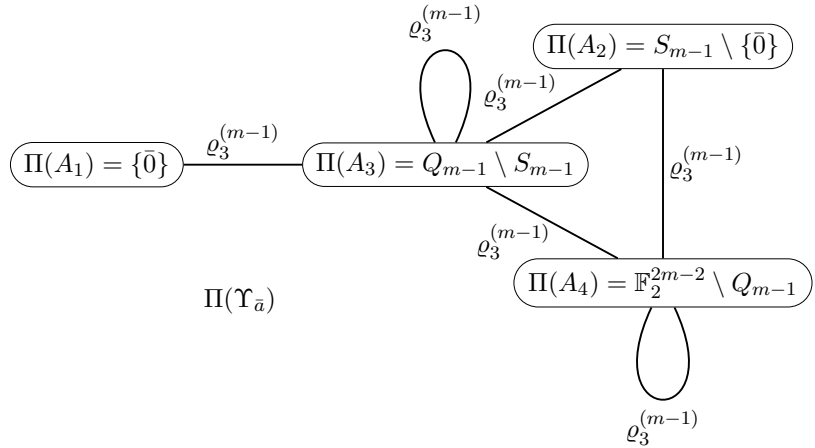
{
\begin{tikzpicture}[node distance=20mm, terminal/.style={rounded rectangle,minimum size=6mm,draw},tight/.style={inner sep=-2pt}]
\node (A3) [terminal,anchor=center] {$\Pi(A_3)=Q_{m-1}\setminus S_{m-1}$};
\node (A2) [terminal, above right= of A3,anchor=center] {$\Pi(A_2)=S_{m-1}\setminus \{\tn\}$};
\node (A4) [terminal, below right= of A3,anchor=center] {$\Pi(A_4)=\F_2^{2m-2}\setminus Q_{m-1}$};
\node (A1) [terminal, left=30mm of A3,anchor=center] {$\Pi(A_1)=\{\tn\}$};
\path[thick, draw]  (A1) edge ["$\varrho_3^{(m-1)}$" inner sep=2pt] (A3);
\path[thick, draw]  (A3) edge ["$\varrho_3^{(m-1)}$"  inner sep=0pt,swap](A4);
\path[thick, draw]  (A3) edge ["$\varrho_3^{(m-1)}$" tight](A2);
\path[thick, draw]  (A2) edge ["$\varrho_3^{(m-1)}$" inner sep=2pt] (A4);
\path[thick, draw] (A3) edge [loop above, min distance=20mm,in=65,out=115,"$\varrho_3^{(m-1)}$"] (A3);
\path[thick, draw] (A4) edge [loop below, min distance=20mm,in=-65,out=-115,swap,"$\varrho_3^{(m-1)}$"] (A4); 
\node at (-3,-2){\ensuremath{\Pi(\Upsilon_{\ta})}}; 
\end{tikzpicture}}	
\]
\[\scalebox{0.9}
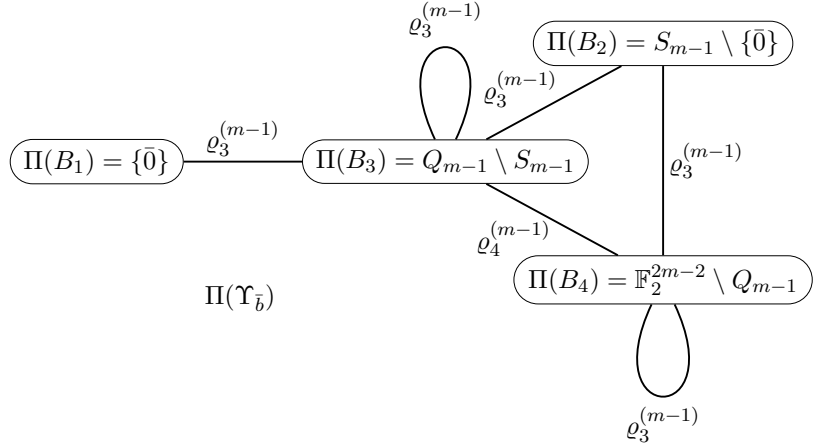
{
\begin{tikzpicture}[node distance=20mm, terminal/.style={rounded rectangle,minimum size=6mm,draw},tight/.style={inner sep=-2pt}]
\node (B3) [terminal,anchor=center] {$\Pi(B_3)=Q_{m-1}\setminus S_{m-1}$};
\node (B2) [terminal, above right= of B3,anchor=center] {$\Pi(B_2)=S_{m-1}\setminus \{\tn\}$};
\node (B4) [terminal, below right= of B3,anchor=center] {$\Pi(B_4)=\F_2^{2m-2}\setminus Q_{m-1}$};
\node (B1) [terminal, left=30mm of B3,anchor=center] {$\Pi(B_1)=\{\tn\}$};
\path[thick, draw]  (B1) edge ["$\varrho_3^{(m-1)}$" inner sep=2pt] (B3);
\path[thick, draw]  (B3) edge ["$\varrho_4^{(m-1)}$"  inner sep=0pt,swap](B4);
\path[thick, draw]  (B3) edge ["$\varrho_3^{(m-1)}$" tight](B2);
\path[thick, draw]  (B2) edge ["$\varrho_3^{(m-1)}$" inner sep=2pt] (B4);
\path[thick, draw] (B3) edge [loop above, min distance=20mm,in=65,out=115,"$\varrho_3^{(m-1)}$"] (B3);
\path[thick, draw] (B4) edge [loop below, min distance=20mm,in=-65,out=-115,swap,"$\varrho_3^{(m-1)}$"] (B4); 
\node at (-3,-2){\ensuremath{\Pi(\Upsilon_{\tb})}}; 
\end{tikzpicture}}
\]
If, e.g., in the first figure there is an edge between, say, $\Pi(A_i)$ and $\Pi(A_j)$ labelled with $\varrho_k^{(m-1)}$, then this means that $(\tu,\tv)\in E(\Upsilon_{\ta})$ 
% \varrho_3^{(m)} \cap (A_i\times A_j)$ 
if and only if $(\Pi(\tu),\Pi(\tv))\in\varrho_k^{(m-1)}$. Also we can read off the diagrams that $\Upsilon_{\ta}$ is isomorphic to $\Gamma^{(m-1)}$ and that a graph isomorphic to $\Upsilon_{\tb}$ can be obtained from $\Gamma^{(m-1)}$ by switching all edges between $Q_{m-1}\setminus S_{m-1}$ and $\F_2^{(2m-2)}\setminus Q_{m-1}$ for edges given by $\varrho_4^{(m-1)}$.

Now we are ready to show that $\Upsilon_{\tb}$ is not isomorphic to $\Upsilon_{\ta}$.  First of all we note that $\Upsilon_{\ta}$, being isomorphic to $\Gamma^{(m-1)}$, is $(3,4)$-regular. Let us count the number of common neighbours of a triangle in $\Gamma^{(m-1)}$. Clearly, this is the number of common neighbours of an edge in the first subconstituent $\Gamma_1^{(m-1)}$. In other words, it is equal to $\bb_m(2\bb_m-1)$.  Let us now consider the triangle of $\Upsilon_{\tb}$ induced by the following three vertices:
\begin{align*}
	\tc_1&=\icol{0\\\tn\\0\\\te_1}, & \tc_2&=\icol{0\\\tn\\0\\\te_2}, & \tc_3&=\icol{0\\\tn\\0\\\te_3}.
\end{align*}
Note that we are able to choose the vertices in this way, since $m\ge 4$, and that $\{\tc_1,\tc_2,\tc_3\}\subseteq B_3$. Thus, the image under $\Pi$ is in $Q_{m-1}\setminus S_{m-1}$. From the diagram of $\Pi(\Upsilon_{\tb})$  we may read that the number of joint neighbours of $\tc_1$, $\tc_2$, and $\tc_3$ is equal to 
\[ 
	 |\{\tv\in Q_{m-1}\mid \forall i:(\Pi(\tc_i),\tv)\in\varrho_3^{(m-1)}\}| + |\{\tv\in \F_2^{2m-2}\setminus Q_{m-1}\mid \forall i:(\Pi(\tc_i),\tv)\in\varrho_4^{(m-1)}\}|,
\]
which can be shown to be equal to $\bb_m(2\bb_m-3/2)$. This shows that $\Upsilon_{\tb}$ is not isomorphic to $\Upsilon_{\ta}$, for all $m\ge 4$. This completes the proof that $\Gamma^{(m)}$ is not $2$-homogeneous. 

It remains to show that $\Gamma_2^{(m)}$ has an intransitive automorphism group. For this we can make use of our computations above. First we use the fact  that  $\Gamma_2^{(m)}(\ta)$ is isomorphic to $\Gamma_2^{(m)}(\tb)$. Second we argue that $\tz$ is a vertex of both graphs. Third we note that $\Upsilon_{\ta}$ is the first subconstituent of $\Gamma_2^{(m)}(\ta)$ with respect to $\tz$ and that $\Upsilon_{\tb}$ is the first subconstituent of $\Gamma_2^{(m)}(\tb)$ with respect to $\tz$. Now, the fact that $\Upsilon_{\ta}$ and $\Upsilon_{\tb}$ are non-isomorphic shows that $\Gamma_2^{(m)}$ contains two different kinds of vertices. In other words, the automorphism group of $\Gamma_2^{(m)}$ has at least two orbits on vertices.
\end{proof}

\section{The Schurian closure of $\Gamma^{(m)}$}
We define the \emph{Schurian closure} of a graph $\Gamma$ to be the relational structure on $V(\Gamma)$ whose basic relations are the orbitals of $\Aut(\Gamma)$. The Schurian closure of a graph gives rise to a so-called coherent configuration. Recall that a \emph{coherent configuration} $\mathcal{C}$ is a finite relational structure $(V, (\varrho_i)_{i\in I})$, such that 
\begin{itemize}
	\item the set $\{\varrho_i\mid i\in I\}$ forms a partition of $V\times V$, 
	\item every $\varrho_i$ is either contained in the diagonal relation $\Delta_V=\{(x,x)\mid x\in V\}$, or it is irreflexive,
	\item every $\varrho_i$ is either symmetric or asymmetric,
	\item for all $i,j,k\in I$ there exists an integer $p_{i,j}^k$, such that for all $(x,y)\in\varrho_k$ we have
	\[
		|\{ z\mid (x,z)\in\varrho_i\land (z,y)\in\varrho_j\}|=p_{i,j}^k.
	\]
\end{itemize} 
The $(p_{i,j}^k)_{i,j,k\in I}$ are called the \emph{structure constants} of $\mathcal{C}$. A coherent configuration $\mathcal{C}$ is called \emph{Schurian} if its relations coincide with the orbitals of its automorphism group (here the automorphism group of $\mathcal{C}$ consists of all permutations of $V$ that preserve each relation $\varrho_i$ where $i\in I$). Note that this is the same as to say  that $\mathcal{C}$, considered as a relational structure, is $2$-homogeneous (i.e., every isomorphism between substructures of cardinality at most $2$ extends to an automorphism). 
If $\mathcal{C}=(V(\Gamma),(\varrho_i)_{i=1,\dots,k})$ is the Schurian closure of $\Gamma$, then it is not hard to see that $\mathcal{C}$ is a Schurian coherent configuration. 

The knowledge of the Schurian closure of $\Gamma^{(m)}$ and, in particular, the knowledge of its structure constants is going to be essential in proving the $(3,5)$-regularity of $\Gamma^{(m)}$.  Our considerations from the previous section suggest that  $\Aut(\Gamma^{(m)})$ has at least $4$ orbitals. In the following we show that the orbitals of $\Aut(\Gamma^{(m)})$ are exactly the relations $\varrho_1^{(m)},\dots,\varrho_4^{(m)}$ that were defined in Theorem~\ref{mainthm}:
\begin{proposition}\label{rankthree}
	The orbitals of $\Aut(\Gamma^{(m)})$ are given by the following binary relations on $\F_2^{2m}$:
	\begin{align*}
		\varrho_1^{(m)} &= \{(\tv,\tw)\mid \tv=\tw\}, & \varrho_2^{(m)} &= \{(\tv,\tw)\mid \tv+\tw\in S_m\setminus\{\tn\}\,\},\\
		\varrho_3^{(m)} &= \{ (\tv,\tw)\mid  \tv+\tw\in Q_m\setminus S_m\}, & \varrho_4^{(m)} &= \{(\tv,\tw)\mid \tv+\tw\notin Q_m\}.
	\end{align*}
\end{proposition}
Before we come to the proof of this Proposition, we need a few auxiliary results:
\begin{lemma}\label{symcond}
	Let $\tu,\tv\in\F_2^m$. Then a symmetric $m\times m$-matrix $S$ over $\F_2$ with zero-diagonal exists such that  $S\tu=\tv$ if and only if  either $\tu=\tv=\tn$ or $\tu\neq\tn$ and $\tu^T\tv=0$. 
\end{lemma}
\begin{proof}
	A symmetric $m\times m$-matrix $S$ with zero-diagonal may be considered as the adjacency matrix of a simple graph $\Gamma$ with vertex set $\{1,\dots,m\}$ and with $i$ connected to $j$ if and only if $S(i,j)=1$.
	Let 
	\begin{align*}
		I_0& \coloneqq\{i\in\{1,\dots,m\}\mid  \tu(i)=0\}, & I_1&\coloneqq \{i\in\{1,\dots,m\}\mid  \tu(i)=1\},\\
		J_0& \coloneqq \{j\in\{1,\dots,m\}\mid  \tv(j)=0\}, & J_1&\coloneqq \{j\in\{1,\dots,m\}\mid  \tv(j)=1\}.				
	\end{align*}
	Then $S\tu=\tv$ if and only if in $\Gamma$  every element of $J_1$ has an odd number of neighbours and every element of $J_0$ has an even number of neighbours in $I_1$, respectively. 
	More detailedly, if we define  
	\begin{align*}
		I_{00} &\coloneqq I_0\cap J_0, & I_{01} &\coloneqq I_0\cap J_1,& I_{10} &\coloneqq  I_1\cap J_0 & I_{11} &\coloneqq  I_1\cap J_1,
	\end{align*} 
	then $S\tu=\tv$ if and only if the  parity of the valencies of the vertices from the $I_{ij}$ to $I_1$ is as depicted in the following diagram:
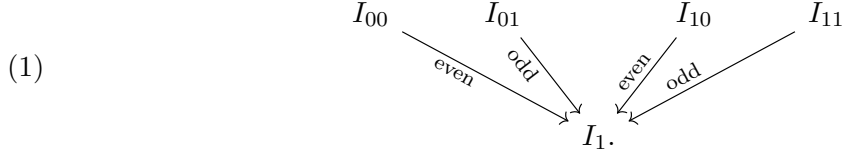
\begin{equation}\label{paritydiag}
	\begin{tikzcd}[column sep=small]   
		I_{00} \arrow[to=I1,"\text{even}" sloped,swap] & & I_{01}\arrow[to=I1,"\text{odd}" sloped,swap] & & I_{10}\arrow[to=I1,"\text{even}" sloped,swap,near start] & &I_{11}\arrow[to=I1,"\text{odd}" sloped,swap]\\[2ex]
		& & & |[alias=I1]|I_1.
	\end{tikzcd}
\end{equation}
	
	``$\Rightarrow$'' 
	We need to show that $\tu^T\tv=0$. This means that $|I_{11}|$ is even. Suppose on the contrary that $|I_{11}|$ is odd. Let us count the number of arcs in the subgraph of $\Gamma$ induced by $I_1$. Since $|I_{11}|$ is odd, there is an odd number of arcs from $I_{11}$ to $I_{1}$. As the number of arcs from $I_{11}$ to $I_{11}$ must be even (by the first theorem of graph theory), the number of arcs from $I_{11}$ to $I_{10}$ must be odd. By symmetry, there is an odd number of arcs from $I_{10}$ to $I_{11}$. As the number of arc from $I_{10}$ to $I_1$ must be even, we conclude that the number of arcs from $I_{10}$ to $I_{10}$ must be odd, a contradiction with the first theorem of graph theory. Hence, the cardinality of $I_{11}$ must be even and thus $\tu^T\tv=0$. 

``$\Leftarrow$'' If $\tu=\tv=\tn$, then we may chose $S=O$. So suppose that $\tu\neq\tn$ and that $\tu^T\tv=0$.  Then $|I_{11}|$ is even. We define a graph $\Gamma$ with vertex set $\{1,\dots,m\}$: The subgraph of $\Gamma$ induced by $I_{11}$ shall be a complete graph. The induced subgraphs $\Gamma(I_{10})$, $\Gamma(I_{01})$, and $\Gamma(I_{00})$ shall have no edge at all. Finally, every vertex from $I_{01}$ shall be connected with exactly one vertex from $I_1$. Clearly, the valencies of the vertices of $\Gamma$ satisfy the parity-conditions from diagram~\eqref{paritydiag}. Thus, if we let $S$ be the adjacency matrix of $\Gamma$, then $S\tu=\tv$.    
\end{proof}
\begin{lemma}\label{Scond}
	Let $A\in\GL(m,2)$, let $S$ be any square matrix of order $m$, and let $\tu,\tv\in\F_2^{2m}$. Then 
	\[
	\begin{pmatrix}
		(A^T)^{-1} & (A^T)^{-1}S\\
		O & A
	\end{pmatrix}
	\begin{pmatrix}
		\tu_1\\
		\tu_2
	\end{pmatrix}=
	\begin{pmatrix}
		\tv_1\\
		\tv_2
	\end{pmatrix} \iff A\tu_2=\tv_2 \text{ and } \tu_1+S\tu_2 = A^T\tv_1.
	\]
\end{lemma}
\begin{proof}
	Clear.
\end{proof}
\begin{lemma}\label{suborbs}
	The group $H_m$ has orbits $\{\tn\}$, $S_m\setminus\{\tn\}$, $Q_m\setminus S_m$, and $\F_2^{2m}\setminus Q_m$.
\end{lemma}
\begin{proof}
	Let $\tv,\tw\in S_m\setminus\{\tn\}$. Then $\tv=\icol{\tv_1\\\tn}$ and $\tw=\icol{\tw_1\\\tn}$. Let $A\in\GL(m,2)$, such that $A\tv_1=\tw_1$. Then
	\[
	\begin{pmatrix}
		A & O \\
		O & (A^T)^{-1}
	\end{pmatrix}
	\begin{pmatrix} 
	\tv_1\\\tn
	\end{pmatrix}	
= \begin{pmatrix}
	\tw_1\\\tn 	
  \end{pmatrix}.
	\] 
	Thus, $\tv$ and $\tw$ are in the same orbit under $H_m$.
	
	Let $\tv, \tw\in Q_m\setminus S_m$. That is $\tv=\icol{\tv_1\\\tv_2}$, $\tv_2\neq\tn$, $\tv_1^T\tv_2=0$, and $\tw=\icol{\tw_1\\\tw_2}$, $\tw_2\neq\tn$, $\tw_1^T\tw_2=0$. Let $A\in\GL(m,2)$, such that $A\tv_2=\tw_2$. Consider $\ta\coloneqq  A^T\tw_1-\tv_1$. We claim that there is a symmetric $m\times m$-matrix $S$ with zero-diagonal, such that $S\tv_2=\ta$. 
	By Lemma~\ref{symcond} we need to show that $\tv_2^T\ta=0$. We compute:
	\begin{equation}\label{v2a}
		\tv_2^TA^T\tw_1 = (A^{-1}\tw_2)^TA^T\tw_1 = \tw_2^T(A^T)^{-1}A^T\tw_1 = \tw_2^T\tw_1=0.
	\end{equation}
	Together with the fact that $\tv_2^T\tv_1=0$, this proves that $\tv_2^T\ta=0$. Let $S$ be a symmetric matrix with zero-diagonal, such that $S\tv_2=\ta$. Then, by Lemma~\ref{Scond}, we have that 
	\[
	\begin{pmatrix}
		(A^T)^{-1} & (A^T)^{-1}S\\
		O & A
	\end{pmatrix}
	\begin{pmatrix}
		\tv_1\\
		\tv_2
	\end{pmatrix}= 
	\begin{pmatrix}
		\tw_1\\
		\tw_2	
	\end{pmatrix}. 
	\]
	
	The case $\tv,\tw\in\F_2^{2m}\setminus Q_m$ is handled in the same way as the previous case. Only the final result in \eqref{v2a} is $1$ and $\tv_2^T\tv_1=1$, thus also in this case $\tv_2^T\ta=0$.
\end{proof}
\begin{proof}[Proof of Proposition~\ref{rankthree}]
	The group $G_m$  acts transitively on $V(\Gamma^{(m)})$. Thus, by Lemma~\ref{suborbs}, $\varrho_1^{(m)},\dots,\varrho_4^{(m)}$ are the orbitals of $G_m$. Since $G_m\le\Aut(\Gamma^{(m)})$, and since by Proposition~\ref{notrankthree} $\Aut(\Gamma^{(m)})$ has at least $4$ orbitals, we conclude that the $\varrho_i^{(m)}$ ($i=1,\dots,4$) are precisely the orbitals of $\Aut(\Gamma^{(m)})$. 	
\end{proof}

Let us denote the Schurian closure $(\F_2^{2m};\,\varrho_1^{(m)},\varrho_2^{(m)},\varrho_3^{(m)},\varrho_4^{(m)})$ of $\Gamma^{(m)}$ by $\mathcal{C}^{(m)}$.
This coherent configuration appeared for the first time in \cite{Iva94}, where also its structure constants $(p_{i,j}^k(m))_{i,j,k\in\{1,2,3,4\}}$ were computed. Here  we give this table once more, using our notations: \\
\begin{minipage}{\linewidth-12pt}% to keep image and caption on one page
\begin{center}\scalebox{1}{$\displaystyle\renewcommand{\arraystretch}{1.3}
\begin{array}{cccccc} 
	& & j=1 & j=2 & j=3 & j=4\\\hline
	& k=1 & 1 & 0 & 0 & 0\\
	& k=2 & 0 & 1 & 0 & 0\\
\raisebox{3ex}[0pt][0pt]{i=1}	& k=3 & 0 & 0 & 1 & 0\\
	& k=4 & 0 & 0 & 0 & 1\\\hline
	& k=1 & 0 & 8\bb_m-1  & 0 & 0\\
	& k=2 & 1 & 8\bb_m-2 & 0 & 0\\
\raisebox{3ex}[0pt][0pt]{i=2}	& k=3 & 0 & 0 & 4\bb_m-1 & 4\bb_m\\
	& k=4 & 0 & 0 & 4\bb_m & 4\bb_m-1\\\hline
	& k=1 & 0 & 0  & 4\bb_m(8\bb_m-1) & 0\\
	& k=2 & 0 & 0 & 4\bb_m(4\bb_m-1) & 16\bb_m^2\\
\raisebox{3ex}[0pt][0pt]{i=3}	& k=3 & 1 & 4\bb_m-1 & 4\bb_m(4\bb_m-1) & 4\bb_m(4\bb_m-1)\\
	& k=4 & 0 & 4\bb_m & 4\bb_m(4\bb_m-1) & 4\bb_m(4\bb_m-1)\\\hline
	& k=1 & 0 & 0  & 0 & 4\bb_m(8\bb_m-1)\\
	& k=2 & 0 & 0 & 16\bb_m^2 & 4\bb_m(4\bb_m-1)\\
\raisebox{3ex}[0pt][0pt]{i=4}	& k=3 & 0 & 4\bb_m & 4\bb_m(4\bb_m-1) & 4\bb_m(4\bb_m-1)\\
	& k=4 & 1 & 4\bb_m-1 & 4\bb_m(4\bb_m-1) & 4\bb_m(4\bb_m-1)\\\hline
\end{array}
$}\end{center}
\captionof{table}{The structure constants $p_{ij}^k(m)$ of $\mathcal{C}^{(m)}$}
\end{minipage}

Next we show that the coherent configuration $\mathcal{C}^{(m)}$, considered merely as a relational structure in the model-theoretic sense, has another remarkable property:
\begin{proposition}
	$\mathcal{C}^{(m)}$, considered as relational structure, is $3$-homogeneous. That is, every isomorphism between substructures of cardinality at most  $3$ extends to an automorphism of $\mathcal{C}^{(m)}$. 
\end{proposition}
\begin{proof}
	We already know that $\mathcal{C}^{(m)}$ is $1$-homogeneous and $2$-homogeneous. Following is a list of isomorphism types of substructures on $3$ elements in $\mathcal{C}^{(m)}$:
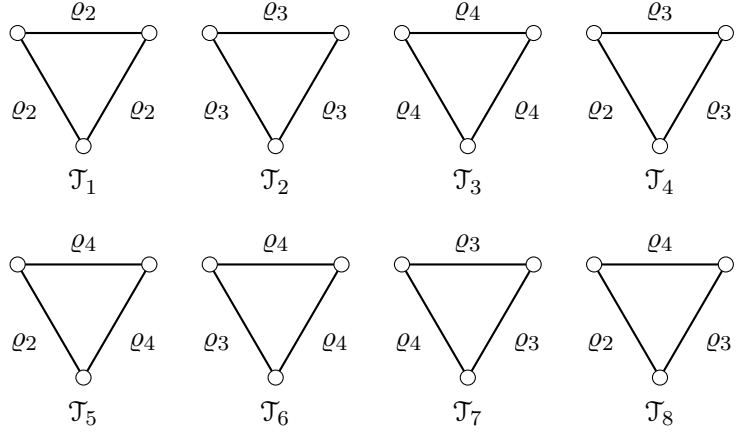
	\[
	\begin{matrix}
	\begin{tikzpicture}[baseline={(current bounding box.center)}]
		\tikzset{VertexStyle/.style = {shape = circle, draw,minimum size = 2pt, inner sep = 2pt}}        
		\SetVertexLabelOut
	    \SetVertexMath
	    \SetVertexLabel
   		\Vertex[a=0*120-90,d=1,NoLabel]{a}
   		\Vertex[a=1*120-90,d=1,NoLabel]{b}
   		\Vertex[a=2*120-90,d=1,NoLabel]{c}
		\Edge[labelstyle={MyLabel}, label=$\varrho_2$](a)(b) 
		\Edge[labelstyle={MyLabel}, label={$\varrho_2$}](b)(c) 
		\Edge[labelstyle={MyLabel,swap}, label={$\varrho_2$}](a)(c)  
	   \end{tikzpicture} & 
	\begin{tikzpicture}[baseline={(current bounding box.center)}]
		\tikzset{VertexStyle/.style = {shape = circle, draw,minimum size = 2pt, inner sep = 2pt}}        
		\SetVertexLabelOut
	    \SetVertexMath
	    \SetVertexLabel
   		\Vertex[a=0*120-90,d=1,NoLabel]{a}
   		\Vertex[a=1*120-90,d=1,NoLabel]{b}
   		\Vertex[a=2*120-90,d=1,NoLabel]{c}
		\Edge[labelstyle={MyLabel}, label=$\varrho_3$](a)(b) 
		\Edge[labelstyle={MyLabel}, label={$\varrho_3$}](b)(c) 
		\Edge[labelstyle={MyLabel,swap}, label={$\varrho_3$}](a)(c)  
	   \end{tikzpicture} & 
	\begin{tikzpicture}[baseline={(current bounding box.center)}]
		\tikzset{VertexStyle/.style = {shape = circle, draw,minimum size = 2pt, inner sep = 2pt}}        
		\SetVertexLabelOut
	    \SetVertexMath
	    \SetVertexLabel
   		\Vertex[a=0*120-90,d=1,NoLabel]{a}
   		\Vertex[a=1*120-90,d=1,NoLabel]{b}
   		\Vertex[a=2*120-90,d=1,NoLabel]{c}
		\Edge[labelstyle={MyLabel}, label=$\varrho_4$](a)(b) 
		\Edge[labelstyle={MyLabel}, label={$\varrho_4$}](b)(c) 
		\Edge[labelstyle={MyLabel,swap}, label={$\varrho_4$}](a)(c)  
	   \end{tikzpicture} & 
	\begin{tikzpicture}[baseline={(current bounding box.center)}]
		\tikzset{VertexStyle/.style = {shape = circle, draw,minimum size = 2pt, inner sep = 2pt}}        
		\SetVertexLabelOut
	    \SetVertexMath
	    \SetVertexLabel
   		\Vertex[a=0*120-90,d=1,NoLabel]{a}
   		\Vertex[a=1*120-90,d=1,NoLabel]{b}
   		\Vertex[a=2*120-90,d=1,NoLabel]{c}
		\Edge[labelstyle={MyLabel}, label=$\varrho_3$](a)(b) 
		\Edge[labelstyle={MyLabel}, label={$\varrho_3$}](b)(c) 
		\Edge[labelstyle={MyLabel,swap}, label={$\varrho_2$}](a)(c)  
	   \end{tikzpicture}\\[6.5ex]
	   \mathcal{T}_1 & \mathcal{T}_2 & \mathcal{T}_3 & \mathcal{T}_4\\[2ex]
	\begin{tikzpicture}[baseline={(current bounding box.center)}]
		\tikzset{VertexStyle/.style = {shape = circle, draw,minimum size = 2pt, inner sep = 2pt}}        
		\SetVertexLabelOut
	    \SetVertexMath
	    \SetVertexLabel
   		\Vertex[a=0*120-90,d=1,NoLabel]{a}
   		\Vertex[a=1*120-90,d=1,NoLabel]{b}
   		\Vertex[a=2*120-90,d=1,NoLabel]{c}
		\Edge[labelstyle={MyLabel}, label=$\varrho_4$](a)(b) 
		\Edge[labelstyle={MyLabel}, label={$\varrho_4$}](b)(c) 
		\Edge[labelstyle={MyLabel,swap}, label={$\varrho_2$}](a)(c)  
	   \end{tikzpicture} &	
	\begin{tikzpicture}[baseline={(current bounding box.center)}]
		\tikzset{VertexStyle/.style = {shape = circle, draw,minimum size = 2pt, inner sep = 2pt}}        
		\SetVertexLabelOut
	    \SetVertexMath
	    \SetVertexLabel
   		\Vertex[a=0*120-90,d=1,NoLabel]{a}
   		\Vertex[a=1*120-90,d=1,NoLabel]{b}
   		\Vertex[a=2*120-90,d=1,NoLabel]{c}
		\Edge[labelstyle={MyLabel}, label=$\varrho_4$](a)(b) 
		\Edge[labelstyle={MyLabel}, label={$\varrho_4$}](b)(c) 
		\Edge[labelstyle={MyLabel,swap}, label={$\varrho_3$}](a)(c)  
	   \end{tikzpicture} &	
	\begin{tikzpicture}[baseline={(current bounding box.center)}]
		\tikzset{VertexStyle/.style = {shape = circle, draw,minimum size = 2pt, inner sep = 2pt}}        
		\SetVertexLabelOut
	    \SetVertexMath
	    \SetVertexLabel
   		\Vertex[a=0*120-90,d=1,NoLabel]{a}
   		\Vertex[a=1*120-90,d=1,NoLabel]{b}
   		\Vertex[a=2*120-90,d=1,NoLabel]{c}
		\Edge[labelstyle={MyLabel}, label=$\varrho_3$](a)(b) 
		\Edge[labelstyle={MyLabel}, label={$\varrho_3$}](b)(c) 
		\Edge[labelstyle={MyLabel,swap}, label={$\varrho_4$}](a)(c)  
	   \end{tikzpicture} &	
	\begin{tikzpicture}[baseline={(current bounding box.center)}]
		\tikzset{VertexStyle/.style = {shape = circle, draw,minimum size = 2pt, inner sep = 2pt}}        
		\SetVertexLabelOut
	    \SetVertexMath
	    \SetVertexLabel
   		\Vertex[a=0*120-90,d=1,NoLabel]{a}
   		\Vertex[a=1*120-90,d=1,NoLabel]{b}
   		\Vertex[a=2*120-90,d=1,NoLabel]{c}
		\Edge[labelstyle={MyLabel}, label=$\varrho_3$](a)(b) 
		\Edge[labelstyle={MyLabel}, label={$\varrho_4$}](b)(c) 
		\Edge[labelstyle={MyLabel,swap}, label={$\varrho_2$}](a)(c)  
	   \end{tikzpicture}\\[6.5ex]
	   \mathcal{T}_5 & \mathcal{T}_6 & \mathcal{T}_7 & \mathcal{T}_8
	   \end{matrix}
	\]
	Following for each type $\mathcal{T}_i$ we consider triples $(\ta,\tb,\tc)$ and $(\tu,\tv,\tw)$ of vertices of $\mathcal{C}^{(m)}$ that induce substructures isomorphic to $\mathcal{T}_i$, such that the mapping $\varphi\colon \ta\mapsto\tu,\, \tb\mapsto\tv,\,\tc\mapsto\tw$ is an isomorphism. Since $\mathcal{C}^{(m)}$ is $2$-homogeneous, in each case, without loss of generality, we may assume that $\ta=\tu=\tn$ and $\tb=\tv$. Throughout the proof we fix the notation 
	\begin{align*}
		\tc_1 &= \begin{pmatrix}
			\tc_1(1)\\\tilde{c}_1
		\end{pmatrix}, & \tc_2 &=\begin{pmatrix}
			\tc_2(1)\\\tilde{c}_2
		\end{pmatrix}, & \tw_1 &=\begin{pmatrix}
			\tw_1(1)\\\tilde{w}_1
		\end{pmatrix}, & \tw_2 &=\begin{pmatrix}
			\tw_2(1)\\\tilde{w}_2
		\end{pmatrix},
	\end{align*}
	for certain $\tilde{c}_i, \tilde{w}_i\in\F_2^{m-1}$ ($i=1,2$). In each case we will find some $A\in\GL(m,2)$ and some symmetric square matrix $S$ of order $m$ with zero-diagonal, such that 
	\begin{align}\label{goal}
	\begin{pmatrix}
		(A^T)^{-1} & (A^T)^{-1}S\\
		O & A
	\end{pmatrix}
	\begin{pmatrix}
		\tb_1\\
		\tb_2
	\end{pmatrix}&=
	\begin{pmatrix}
		\tv_1\\
		\tv_2
	\end{pmatrix}, &
	\begin{pmatrix}
		(A^T)^{-1} & (A^T)^{-1}S\\
		O & A
	\end{pmatrix}
	\begin{pmatrix}
		\tc_1\\
		\tc_2
	\end{pmatrix}&=
	\begin{pmatrix}
		\tw_1\\
		\tw_2
	\end{pmatrix}.
	\end{align}	
	
	``about $\mathcal{T}_1$:'' Without loss of generality we may assume that  $\tb=\tv=\icol{\te_1\\\tn}$. Moreover, $\tc_2=\tw_2=\tn$.  Choose an $\hat{A}\in\GL(m,2)$ that fixes $\te_1$ and that maps $\tc_1$ to $\tw_1$ (such an $\hat{A}$ exists, since $\tc_1\neq\tn$ and because $\GL(m,2)$ acts $2$-transitively on non-zero vectors). Then with $A\coloneqq (\hat{A}^T)^{-1}$, and $S\coloneqq O$ we have that \eqref{goal} is satisfied.
	
	``about $\mathcal{T}_2$:'' Without loss of generality we may assume that, $\tb=\tv=\icol{\tn\\\te_1}$, $\tc_1(1)=\tw_1(1)=0$, $\tc_2,\tw_2\notin\{\tn,\te_1\}$, and $\tc_1^T\tc_2=\tw_1^T\tw_2=0$. Note that 
	then $\icol{\tilde{c}_1\\\tilde{c}_2},\icol{\tilde{w}_1\\\tilde{w}_2}\in Q_{m-1}\setminus S_{m-1}$. By Lemma~\ref{suborbs} there exists an element of $H_{m-1}$ that maps $\icol{\tilde{c}_1\\\tilde{c}_2}$ to $\icol{\tilde{w}_1\\\tilde{w}_2}$. Thus, by Lemma~\ref{HM}, there exist  $\tilde{A}\in\GL(m-1,2)$ and a symmetric square matrix $\tilde{S}$ of order $m-1$ with zero-diagonal, such that 
	\[
		\begin{pmatrix}
			(\tilde{A}^T)^{-1} & (\tilde{A}^T)^{-1}\tilde{S}\\
			O & \tilde{A}
		\end{pmatrix}
		\begin{pmatrix}
			\tilde{c}_1\\
			\tilde{c}_2	
		\end{pmatrix}=
		\begin{pmatrix}
			\tilde{w}_1\\
			\tilde{w}_2
		\end{pmatrix}.
	\]
	Let $\tilde{x}\in\F_2^{m-1}$, such that $\tilde{x}^T\tilde{c}_2=\tc_2(1)+\tw_2(1)$ (such an $\tilde{x}$ exists because $\tilde{c}_2\neq\tn$). Define
	\begin{align*}
		A&\coloneqq 
		\begin{pmatrix}
		1 & \tilde{x}^T\\
		\tn & \tilde{A}
		\end{pmatrix}, & S &\coloneqq 
		\begin{pmatrix}
		0 & \tn^T\\
		\tn & \tilde{S}
	\end{pmatrix}.
	\end{align*}
	Then, using Lemma~\ref{Scond}, it can be  checked that \eqref{goal} is satisfied.
	
	``about $\mathcal{T}_3$:'' Without loss of generality we may assume that  $\tb=\tv=\icol{\te_1\\\te_1}$, $\tc_1(1)\neq\tc_2(1)$, $\tw_1(1)\neq\tw_2(1)$,  and $\tc_1^T\tc_2=\tw_1^T\tw_2=1$. 
	Observe that $\tilde{c}_1^T\tilde{c}_2=\tilde{w}_1^T\tilde{w}_1=1$. 
	Let $\tilde{x}\in\F_2^{m-1}$, such that $\tilde{x}^T\tilde{c}_2=\tc_1(1)+\tw_1(1)$ (such an $\tilde{x}$ exists because $\tilde{c}_2\neq\tn$).
	Note that
	\[
		(1+\tc_1(1)+\tw_1(1))\tilde{x}^T\tilde{c}_2 = (1+\tilde{x}^T\tilde{c}_2)\tilde{x}^T\tilde{c}_2 = 0.
	\]
	Thus, by Lemma~\ref{suborbs} together with Lemma~\ref{HM}, there exists $\tilde{A}\in\GL(m-1,2)$ and a symmetric square matrix $\tilde{S}$ of order $m-1$ with zero-diagonal, such that 
	\[
		\begin{pmatrix}
			(\tilde{A}^T)^{-1} & (\tilde{A}^T)^{-1}\tilde{S}\\
			O & \tilde{A}
		\end{pmatrix}
		\begin{pmatrix}
			\tilde{c}_1+(1+\tc_1(1)+\tw_1(1))\tilde{x}\\
			\tilde{c}_2	
		\end{pmatrix}=
		\begin{pmatrix}
			\tilde{w}_1\\
			\tilde{w}_2
		\end{pmatrix}.
	\]
 Define
	\begin{align*}
		A&\coloneqq 
		\begin{pmatrix}
		1 & \tilde{x}^T\\
		\tn & \tilde{A}
		\end{pmatrix}, & S &\coloneqq 
		\begin{pmatrix}
		0 & \tilde{x}^T\\
		\tx & \tilde{S}
	\end{pmatrix}.
	\end{align*}
	Then, using Lemma~\ref{Scond}, it can be verified that  \eqref{goal} is satisfied.

	``about $\mathcal{T}_4$:'' Without loss of generality we may assume that $\tb=\tv=\icol{\tn\\\te_1}$, $\tc_1(1)=\tw_1(1)=0$,  $\tc_1\neq\tn$, $\tw_1\neq\tn$, and $\tc_2=\tw_2=\tn$. 
 Let $\tilde{A}\in\GL(m-1,2)$ such that $\tilde{A}\tilde{c}_1=\tilde{w}_1$. Then, with
\[
A\coloneqq 
	\begin{pmatrix}
	1 & \tn^T\\
	\tn & (\tilde{A}^T)^{-1}	
	\end{pmatrix}
\]
and with $S\coloneqq O$, it can be checked that \eqref{goal} is satisfied.

	``about $\mathcal{T}_5$:'' Without loss of generality we may assume that  $\tb=\tv=\icol{\te_1\\\te_1}$, $\tc_1(1)=\tw_1(1)=0$,  $\tc_1\neq\tn$, $\tw_1\neq\tn$, and $\tc_2=\tw_2=\tn$. 
 Let $\tilde{A}\in\GL(m-1,2)$ such that $\tilde{A}\tilde{c}_1=\tilde{w}_1$. Then, with
\[
A\coloneqq 
	\begin{pmatrix}
	1 & \tn^T\\
	\tn & (\tilde{A}^T)^{-1}	
	\end{pmatrix}
\]
and with $S\coloneqq O$ it can be checked that \eqref{goal} is satisfied.

	``about $\mathcal{T}_6$:'' Without loss of generality we may assume that  $\tb=\tv=\icol{\tn\\\te_1}$, $\tc_1(1)=\tw_1(1)=0$, and $\tc_1^T\tc_2=\tw_1^T\tw_2=1$. Observe that $\icol{\tilde{c}_1\\\tilde{c}_2},\icol{\tilde{w}_1\\\tilde{w}_2}\in \F_2^{2m-2}\setminus Q_{m-1}$. By Lemma~\ref{suborbs} in conjunction with Lemma~\ref{HM} there exists $\tilde{A}\in\GL(m-1,2)$ and a symmetric square matrix $\tilde{S}$ of order $m-1$ with zero-diagonal, such that 
	\[
		\begin{pmatrix}
			(\tilde{A}^T)^{-1} & (\tilde{A}^T)^{-1}\tilde{S}\\
			O & \tilde{A}
		\end{pmatrix}
		\begin{pmatrix}
			\tilde{c}_1\\
			\tilde{c}_2	
		\end{pmatrix}=
		\begin{pmatrix}
			\tilde{w}_1\\
			\tilde{w}_2
		\end{pmatrix}.
	\]
	Let $\tilde{x}\in\F_2^{m-1}$, such that $\tilde{x}^T\tilde{c}_2=\tc_2(1)+\tw_2(1)$ (such an $\tilde{x}$ exists because $\tilde{c}_2\neq\tn$). Define
	\begin{align*}
		A&\coloneqq 
		\begin{pmatrix}
		1 & \tilde{x}^T\\
		\tn & \tilde{A}
		\end{pmatrix}, & S &\coloneqq 
		\begin{pmatrix}
		0 & \tn^T\\
		\tn & \tilde{S}
	\end{pmatrix}.
	\end{align*}
	Then, using Lemma~\ref{Scond}, it can be checked that \eqref{goal} is satisfied.
	
	``about $\mathcal{T}_7$:'' Without loss of generality we may assume that  $\tb,\tv=\icol{\tn\\\te_1}$, $\tc_1(1)=\tw_1(1)=1$,   and $\tc_1^T\tc_2=\tw_1^T\tw_2=1$, $\tc_2\neq\te_1$, $\tw_2\neq\te_1$. 
	Then  $\tilde{c}_2\neq\tn$, $\tilde{w}_2\neq\tn$. Let $\tilde{x}\in\F_2^{m-1}\setminus\{\tilde{c}_2\}$, such that $\tilde{x}^T\tilde{c}_2=\tc_2(1)+\tw_2(1)$. Then $(\tilde{c}_1+\tilde{x})^T\tilde{c}_2=\tilde{w}_1^T\tilde{w}_2$. Thus, $\icol{\tilde{c}_1+\tilde{x}\\\tilde{c}_2}$ and $\icol{\tilde{w}_1\\\tilde{w}_2}$ are either both in $\F_2^{m-2}\setminus Q_{m-1}$ or both in $Q_{m-1}\setminus S_{m-1}$. By Lemma~\ref{suborbs} together with Lemma~\ref{HM} there exists $\tilde{A}\in\GL(m-1,2)$ and a symmetric square matrix $\tilde{S}$ of order $m-1$ with zero-diagonal, such that 
	\[
		\begin{pmatrix}
			(\tilde{A}^T)^{-1} & (\tilde{A}^T)^{-1}\tilde{S}\\
			O & \tilde{A}
		\end{pmatrix}
		\begin{pmatrix}
			\tilde{c}_1+\tilde{x}\\
			\tilde{c}_2	
		\end{pmatrix}=
		\begin{pmatrix}
			\tilde{w}_1\\
			\tilde{w}_2
		\end{pmatrix}.
	\]	
	With 	\begin{align*}
	A&\coloneqq \begin{pmatrix}
		1 & \tilde{x}^T\\
		\tn & \tilde{A}
	\end{pmatrix}, &
	S\coloneqq \begin{pmatrix}
		0 & \tn^T\\
		\tn & \tilde{S}
	\end{pmatrix},
	\end{align*} 
	using Lemma~\ref{Scond}, it can be checked that \eqref{goal} is satisfied. 

	``about $\mathcal{T}_8$:'' Without loss of generality we may assume that  $\tb=\tv=\icol{\tn\\\te_1}$, $\tc_1(1)=\tw_1(1)=1$, and $\tc_2=\tw_2=\tn$. 
 Let $\tilde{A}\in\GL(m-1,2)$, such that $\tilde{A}\tilde{c}_1=\tilde{w}_1$. Then with
	\[
	A\coloneqq \begin{pmatrix}
		1 & \tn^T\\
		\tn & (\tilde{A}^T)^{-1}
	\end{pmatrix}
	\]
	and $S\coloneqq O$ it can be checked that \eqref{goal} is satisfied. 
\end{proof}

\section{$(3,5)$-regularity of the graphs $\Gamma^{(m)}$}
The proof of the $(3,5)$-regularity of $\Gamma^{(m)}$ hinges on a recent result reducing the number of  graph types to be checked for regularity. Let us repeat the relevant details:
\begin{definition}
    Let $\bbT=(\Delta,\iota,\Theta)$ be a graph type. Suppose $\Theta=(T,E)$. Let $M\subseteq T$ be the image of $\iota$. Then we define the \emph{closure} $\Cl(\bbT)$ to be the graph with vertex set $T$ and with arc set $E\cup \{(u,v)\mid u,v\in M,\, u\neq v\}$.
\end{definition}
\begin{theorem}[{\cite[Corollary~3.41]{Pec14}}]
	A graph $\Gamma$ is $(m,n+1)$-regular if and only if it is $(m,n)$-regular and $\bbT$-regular for all graph types $\bbT$ of order $(m,n+1)$ for which $\Cl(\bbT)$ is $(m+1)$-connected. 
\end{theorem}
We know that $\Gamma^{(m)}$ is $(3,4)$-regular. Next we enumerate all graph types of order $(3,5)$ whose closure is $4$-connected. The only $4$-connected graph of order $5$ is $K_5$. Thus the graph types of order $(3,5)$ with $4$-connected closure are:
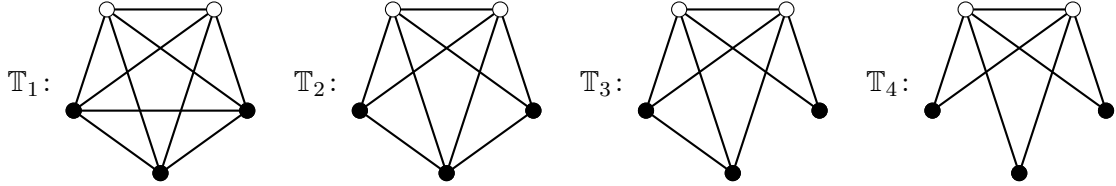
\begin{align*}
	\bbT_1\colon\, &   
	\begin{tikzpicture}[rotate=-162,scale=0.4,baseline={(current bounding box.center)}]
		\tikzset{VertexStyle/.style = {shape = circle, draw,minimum size = 2pt, inner sep = 2pt}}        
		\SetVertexMath
        \Vertices[unit=3,NoLabel]{circle}{A,B,C,D,E}
        \Edges(A,E,B,D,C,E,D,A)
        \AddVertexColor{black}{A,B,C}
        \Edges(A,B,C,A)
    \end{tikzpicture} & \bbT_2\colon\, &
	\begin{tikzpicture}[rotate=-162,scale=0.4,baseline={(current bounding box.center)}]
		\tikzset{VertexStyle/.style = {shape = circle, draw,minimum size = 2pt, inner sep = 2pt}}        
		\SetVertexMath
        \Vertices[unit=3,NoLabel]{circle}{A,B,C,D,E}
        \Edges(A,E,B,D,C,E,D,A)
        \AddVertexColor{black}{A,B,C}
        \Edges(A,B,C)
    \end{tikzpicture} & \bbT_3\colon\, &
	\begin{tikzpicture}[rotate=-162,scale=0.4,baseline={(current bounding box.center)}]
		\tikzset{VertexStyle/.style = {shape = circle, draw,minimum size = 2pt, inner sep = 2pt}}        
		\SetVertexMath
        \Vertices[unit=3,NoLabel]{circle}{A,B,C,D,E}
        \Edges(A,E,B,D,C,E,D,A)
        \AddVertexColor{black}{A,B,C}
        \Edges(A,B)
    \end{tikzpicture} & \bbT_4\colon\, &
	\begin{tikzpicture}[rotate=-162,scale=0.4,baseline={(current bounding box.center)}]
		\tikzset{VertexStyle/.style = {shape = circle, draw,minimum size = 2pt, inner sep = 2pt}}        
		\SetVertexMath
        \Vertices[unit=3,NoLabel]{circle}{A,B,C,D,E}
        \Edges(A,E,B,D,C,E,D,A)
        \AddVertexColor{black}{A,B,C}
    \end{tikzpicture}
\end{align*}
In other words, from the (up to isomorphism) $148$ graph types of order $(3,5)$ only $4$ have to be checked in order to prove that $\Gamma^{(m)}$ is $(3,5)$-regular.

In the course of the proof of $(3,5)$-regularity of $\Gamma^{(m)}$ the following classical graph theoretical concept will play a crucial role:
\begin{definition}
	Let $\Gamma$ be a graph. An \emph{equitable partition} of $\Gamma$ is an ordered partition $(M_1,\dots,M_n)$ of $V(\Gamma)$, such that for all $i,j\in\{1,\dots,n\}$ there exists a non-negative integer $a_{ij}$, such that for all $v\in M_i$ the number of neighbours of $v$ in $M_j$ is equal to $a_{ij}$. The matrix $(a_{ij})_{i,j=1}^n$ is called the \emph{partition matrix} of the equitable partition. 
\end{definition} 

\subsubsection*{Proof of $\bbT_1$-regularity}
To fix notation, let $\bbT_1=(\Delta,\iota,\Theta)$, where $V(\Delta)=\{a,b,c\}$ and $E(\Delta)=\{(a,b),(a,c),(b,a),(b,c),(c,a),(c,b)\}$. Let $\kappa_1,\kappa_2\colon\Delta\injto\Gamma^{(m)}$ be two embeddings. In particular, suppose $\kappa_1\colon a\mapsto \tu,\,b\mapsto \tv,\,c\mapsto \tw$ and $\kappa_2\colon a\mapsto\tx,\, b\mapsto\ty,\, c\mapsto\tz$. Then, since $\mathcal{C}^{(m)}$ is $3$-homogeneous, there exists an automorphism $\varphi$ of $\mathcal{C}^{(m)}$, such that $\varphi\colon \tu\mapsto\tx,\,\tv\mapsto\ty,\,\tw\mapsto\tz$. In other words, $\kappa_2=\varphi\circ\kappa_1$. In particular, since $\Aut(\mathcal{C}^{(m)})=\Aut(\Gamma^{(m)})$, we obtain that  $\#(\Gamma^{(m)},\bbT_1,\kappa_1)=\#(\Gamma^{(m)},\bbT_1,\kappa_2)$. Since $\kappa_1$ and $\kappa_2$ were chosen arbitrarily, we conclude that $\Gamma^{(m)}$ is $\bbT_1$-regular.

\subsubsection*{Proof of $\bbT_2$-regularity}
Let the graph type $\bbT_2=(\Delta,\iota,\Theta)$ be given by the following labeled diagram:
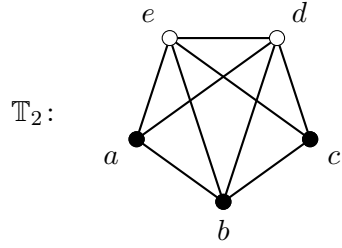
\[
	\bbT_2\colon\quad\begin{tikzpicture}[scale=0.4,baseline={(current bounding box.center)}]
		\tikzset{VertexStyle/.style = {shape = circle, draw,minimum size = 2pt, inner sep = 2pt}}        
		\SetVertexMath
		\SetVertexLabelOut
   		\Vertex[a=0*72-90,d=3,Lpos=0*72-90,Ldist=0,L=b]{B}
   		\Vertex[a=1*72-90,d=3,Lpos=1*72-90,Ldist=0,L=c]{C}
   		\Vertex[a=2*72-90,d=3,Lpos=2*72-90,Ldist=0,L=d]{D}
   		\Vertex[a=3*72-90,d=3,Lpos=3*72-90,Ldist=0,L=e]{E}
   		\Vertex[a=4*72-90,d=3,Lpos=4*72-90,Ldist=0,L=a]{A}
%        \Vertices[unit=3,NoLabel]{circle}{A,B,C,D,E}
        \Edges(A,E,B,D,C,E,D,A)
        \AddVertexColor{black}{A,B,C}
        \Edges(A,B,C)
    \end{tikzpicture}
\]
That is, $V(\Delta)=\{a,b,c\}$, $V(\Theta)=\{a,b,c,d,e\}$, and $\iota\colon\Delta\injto\Theta$ is the identical embedding. Let $\kappa\colon\Delta\injto\Gamma^{(m)}$. Let $\bbT$ be the graph type obtained from $\bbT_2$ by removing the vertex $b$. That is, $\bbT$ has the following diagram:
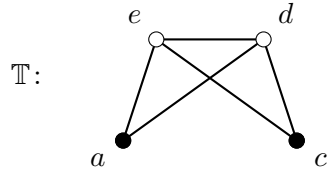
\[
	\bbT\colon\quad \begin{tikzpicture}[scale=0.4,baseline={(current bounding box.center)}]
		\tikzset{VertexStyle/.style = {shape = circle, draw,minimum size = 2pt, inner sep = 2pt}}        
		\SetVertexMath
		\SetVertexLabelOut
   		\Vertex[a=1*72-90,d=3,Lpos=1*72-90,Ldist=0,L=c]{C}
   		\Vertex[a=2*72-90,d=3,Lpos=2*72-90,Ldist=0,L=d]{D}
   		\Vertex[a=3*72-90,d=3,Lpos=3*72-90,Ldist=0,L=e]{E}
   		\Vertex[a=4*72-90,d=3,Lpos=4*72-90,Ldist=0,L=a]{A}
        \Edges(A,E,C,D,A)
        \Edge(E)(D)
        \AddVertexColor{black}{A,C}
    \end{tikzpicture}
\]
Let $\tv\coloneqq \kappa(b)$. Observe, that $\#(\Gamma^{(m)},\bbT_2,\kappa)=\#(\Gamma_1^{(m)}(\tv),\bbT,\kappa\restr_{\{a,c\}})$. Since $\Gamma_1^{(m)}(\tv)$ is $(2,4)$-regular, it is in particular $\bbT$-regular. Thus, since $\Aut(\Gamma^{(m)})$ is transitive, we obtain that $\Gamma^{(m)}$ is $\bbT_2$-regular.

\subsubsection*{Proof of $\bbT_3$-regularity}
Let the graph type $\bbT_3=(\Delta,\iota,\Theta)$ be given by the following labelled diagram:
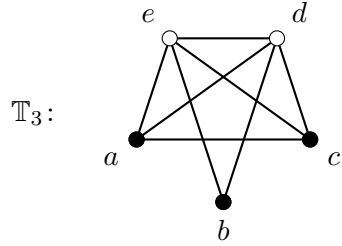
\[
	\bbT_3\colon\quad\begin{tikzpicture}[scale=0.4,baseline={(current bounding box.center)}]
		\tikzset{VertexStyle/.style = {shape = circle, draw,minimum size = 2pt, inner sep = 2pt}}        
		\SetVertexMath
		\SetVertexLabelOut
   		\Vertex[a=0*72-90,d=3,Lpos=0*72-90,Ldist=0,L=b]{B}
   		\Vertex[a=1*72-90,d=3,Lpos=1*72-90,Ldist=0,L=c]{C}
   		\Vertex[a=2*72-90,d=3,Lpos=2*72-90,Ldist=0,L=d]{D}
   		\Vertex[a=3*72-90,d=3,Lpos=3*72-90,Ldist=0,L=e]{E}
   		\Vertex[a=4*72-90,d=3,Lpos=4*72-90,Ldist=0,L=a]{A}
        \Edges(A,E,B,D,C,E,D,A)
        \AddVertexColor{black}{A,B,C}
        \Edges(A,C)
    \end{tikzpicture}
\]
That is, $V(\Delta)=\{a,b,c\}$, $V(\Theta)=\{a,b,c,d,e\}$, and $\iota\colon\Delta\injto\Theta$ is the identical embedding. 
%Let $\kappa\colon\Delta\injto\Gamma^{(m)}$.

Up to symmetries of $\Gamma^{(m)}$ and of $\bbT_3$ there are two kinds of embeddings of $\Delta$
into $\Gamma^{(m)}$. They are distinguished by their induced image in $\mathcal{C}^{(m)}$: Let $\kappa\colon\Delta\injto\Gamma^{(m)}$. Let us denote $\kappa(a)=:\tu$, $\kappa(b)=:\tv$ and $\kappa(c)=:\tw$. Then we have that $\#(\Gamma^{(m)},\bbT_3,\kappa)$ is equal to the number of arcs in the subgraph of $\Gamma^{(m)}$ induced by the set of joint neighbours of $\tu$, $\tv$, and $\tw$. 

Note that $\{\tu,\tv,\tw\}$ induce one of the following two subcolorgraphs in $\mathcal{C}^{(m)}$:
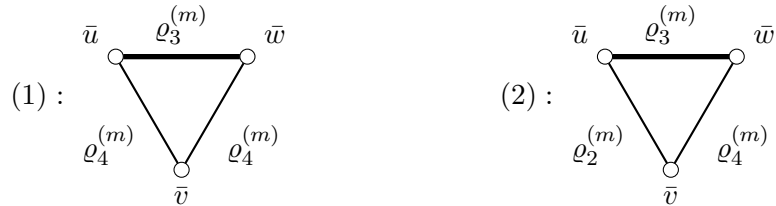
	\begin{align*}
		(1) &:	
   		\begin{tikzpicture}[baseline={(current bounding box.center)}]
		\tikzset{VertexStyle/.style = {shape = circle, draw,minimum size = 2pt, inner sep = 2pt}}        
		\SetVertexLabelOut
	    \SetVertexMath
	    \SetVertexLabel
   		\Vertex[a=0*120-90,d=1,Lpos=0*120-90,Ldist=0,L=\bar{v}]{c}
   		\Vertex[a=1*120-90,d=1,Lpos=1*120-90,Ldist=0,L=\bar{w}]{a}
   		\Vertex[a=2*120-90,d=1,Lpos=2*120-90,Ldist=0,L=\bar{u}]{z}
		\Edge[lw=2pt,labelstyle={MyLabel,swap}, label=$\varrho_3^{(m)}$](z)(a) 
		\Edge[labelstyle={MyLabel}, label={$\varrho_4^{(m)}$}](z)(c) 
		\Edge[labelstyle={MyLabel,swap}, label={$\varrho_4^{(m)}$}](a)(c) 
	   \end{tikzpicture} 
	   & (2) &:
   		\begin{tikzpicture}[baseline={(current bounding box.center)}]
		\tikzset{VertexStyle/.style = {shape = circle, draw,minimum size = 2pt, inner sep = 2pt}}        
		\SetVertexLabelOut
	    \SetVertexMath
	    \SetVertexLabel
   		\Vertex[a=0*120-90,d=1,Lpos=0*120-90,Ldist=0,L=\bar{v}]{c}
   		\Vertex[a=1*120-90,d=1,Lpos=1*120-90,Ldist=0,L=\bar{w}]{a}
   		\Vertex[a=2*120-90,d=1,Lpos=2*120-90,Ldist=0,L=\bar{u}]{z}
		\Edge[lw=2pt,labelstyle={MyLabel,swap}, label=$\varrho_3^{(m)}$](z)(a) 
		\Edge[labelstyle={MyLabel}, label={$\varrho_2^{(m)}$}](z)(c) 
		\Edge[labelstyle={MyLabel,swap}, label={$\varrho_4^{(m)}$}](a)(c) 
	   \end{tikzpicture} 
	\end{align*}

Let us start with the first kind. Since $\mathcal{C}^{(m)}$ is $3$-homogeneous, without loss of generality we can assume that $\tu=\icol{\tn\\\tn}$, $\tw=\icol{\tn\\\te_m}$, and $\tv=\icol{\te_1\\\te_1}$. 
\begin{lemma}\label{vert2}
	Let $\bar{x}\in \F_2^{2m}$. Then $\bar{x}$ is a  joint neighbour of $\{\bar{u},\bar{v},\bar{w}\}$ in $\Gamma^{(m)}$ if and only if 
\begin{enumerate}
	\item $\bar{x}_2\neq\bar{0}$,
	\item $\bar{x}_1^T\bar{x}_2 = 0$,
	\item $\tx_1(m)=0$,
	\item $\tx_1(1)\neq\tx_2(1)$,
	\item $\tx_2(m)=1 \Longrightarrow (\tx_2(1),\dots,\tx_2(m-1))\neq\tn^T$,
	\item $\tx_2(1)=1 \Longrightarrow (\tx_2(2),\dots,\tx_2(m))\neq\tn^T$.
\end{enumerate}	
\end{lemma}
\begin{proof}
	Clear.
\end{proof}
In the following, by $M$ we will denote the set of joint neighbours of $\{\tu,\tv,\tw\}$ in $\Gamma^{(m)}$. Using  Lemma~\ref{vert2} we partition $M$ into $6$ natural classes:
\begin{align*}
	M_1 &= \{\tx\in M\mid \tx_1(1)=1, \tx_2(1)=0, \tx_2(m)=0\},\\
	M_2 &= \{\tx\in M\mid \tx_1(1)=1, \tx_2(1)=0, \tx_2(m)=1\},\\
	M_3 &= \{\tx\in M\mid \tx_1(1)=0, \tx_2(1)=1, \tx_2(m)=0\},\\
	M_4 &= \{\tx\in M\mid \tx_1(1)=0, \tx_2(1)=1, \tx_2(m)=1, (\tx_2(2),\dots,\tx_2(m-1))\neq \tn^T\},\\
	M_5 &= \{\tx\in M\mid \tx_1(1)=0, \tx_1\neq\tn, \tx_2(1)=1, \tx_2(m)=1, (\tx_2(2),\dots,\tx_2(m-1))= \tn^T\},\\ 
	M_6 &= \{\tx\in M\mid \tx_1=\tn, \tx_2=\te_1+\te_m\}.
\end{align*}
We claim that $(M_1,\dots,M_6)$ is an equitable partition of $\langle M\rangle_{\Gamma^{(m)}}$. For the proof of this claim consider now the projection 
\begin{equation}\label{projection}
\Pi\colon \F_2^{2m}\epito\F_2^{2m-4}:\quad 
 \begin{pmatrix}
 	\tx_1\\
 	\tx_2
 \end{pmatrix}\mapsto
 \begin{pmatrix}
 	\tilde{x}_1\\
 	\tilde{x}_2
 \end{pmatrix}, 
\end{equation}
where $\tilde{x}_i$ is the unique element of $\F_2^{m-2}$, such that $\tx_i=\icol{\tx_i(1)\\\tilde{x}_i\\\tx_i(m)}$ (where $i=1,2$). Observe that for each $i\in\{1,\dots,6\}$ we have that $\Pi\restr_{M_i}$ is one-to-one. Routine computations show that the projection of $\langle M\rangle_{\Gamma^{(m)}}$ in $\mathcal{C}^{(m-2)}$ looks as follows:
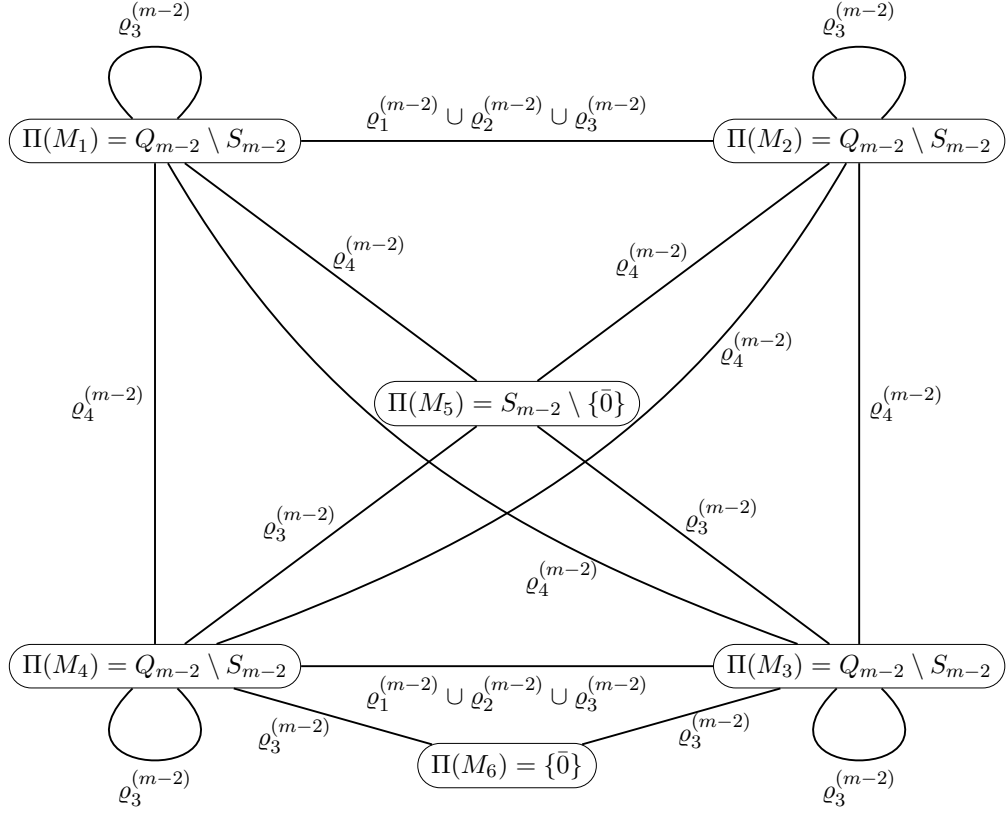
\[\scalebox{0.9}{
\begin{tikzpicture}[node distance=50mm, terminal/.style={rounded rectangle,minimum size=6mm,draw},tight/.style={inner sep=-2pt}]
\node (M5) [terminal,anchor=center] {$\Pi(M_5)=S_{m-2}\setminus\{\tn\}$};
\node (M1) [terminal, above left= of M5,anchor=center] {$\Pi(M_1)=Q_{m-2}\setminus S_{m-2}$};
\node (M2) [terminal, above right= of M5,anchor=center] {$\Pi(M_2)=Q_{m-2}\setminus S_{m-2}$};
\node (M3) [terminal, below right= of M5,anchor=center] {$\Pi(M_3)=Q_{m-2}\setminus S_{m-2}$};
\node (M4) [terminal, below left= of M5,anchor=center] {$\Pi(M_4)=Q_{m-2}\setminus S_{m-2}$};
\node (M6) [terminal, below= 50mm of M5,anchor=center] {$\Pi(M_6)=\{\tn\}$};
\path[thick, draw]  (M1) edge ["$\varrho_4^{(m-2)}$" inner sep=0pt] (M5);
\path[thick, draw]  (M1) edge [out=-60,in=160, "$\varrho_4^{(m-2)}$"  inner sep=-1pt,near end,swap](M3);
\path[thick, draw]  (M1) edge [swap, "$\varrho_4^{(m-2)}$"](M4);
\path[thick, draw]  (M1) edge ["$\varrho_1^{(m-2)}\cup\varrho_2^{(m-2)}\cup\varrho_3^{(m-2)}$"](M2);
\path[thick, draw]  (M2) edge [out=-120,in=20, "$\varrho_4^{(m-2)}$" tight, near start](M4);
\path[thick, draw]  (M2) edge [swap, "$\varrho_4^{(m-2)}$" inner sep=-3pt] (M5);
\path[thick, draw]  (M2) edge ["$\varrho_4^{(m-2)}$"](M3);
\path[thick, draw]  (M3) edge [swap,  "$\varrho_3^{(m-2)}$" inner sep=0pt] (M5);
\path[thick, draw]  (M3) edge ["$\varrho_3^{(m-2)}$" tight](M6);
\path[thick, draw]  (M3) edge ["$\varrho_1^{(m-2)}\cup\varrho_2^{(m-2)}\cup\varrho_3^{(m-2)}$"] (M4);
\path[thick, draw]  (M4) edge ["$\varrho_3^{(m-2)}$" inner sep=-3pt] (M5);
\path[thick, draw]  (M4) edge [swap,"$\varrho_3^{(m-2)}$" inner sep=0pt](M6);
\path[thick, draw] (M1) edge [loop above, min distance=20mm,in=45,out=135,"$\varrho_3^{(m-2)}$"] (M1);
\path[thick, draw] (M2) edge [loop above, min distance=20mm,in=45,out=135,"$\varrho_3^{(m-2)}$"] (M2); 
\path[thick, draw] (M3) edge [loop below, min distance=20mm,in=-45,out=-135,swap, "$\varrho_3^{(m-2)}$"] (M3); 
\path[thick, draw] (M4) edge [loop below, min distance=20mm,in=-45,out=-135,swap, "$\varrho_3^{(m-2)}$"] (M4);   
\end{tikzpicture}}
\]
To be more precise, an edge in the above given diagram from $\Pi(M_i)$ to $\Pi(M_j)$ labelled  by a relation $\sigma$ means that for all $\tx\in M_i$, $\ty\in M_j$, we have that $(\tx,\ty)\in E(\langle M\rangle_{\Gamma^{(m)}})$ if and only if $(\Pi(\tx),\Pi(\ty))\in\sigma$. An immediate consequence of this observation is that $(M_1,\dots, M_6)$ is an equitable partition of $\langle M\rangle_{\Gamma^{(m)}}$, as was claimed before. Moreover, its partition matrix is given by: 
\[
\begin{pmatrix}\def\arraystretch{1.5}
 p_{33}^3 & p_{31}^3+p_{32}^3+p_{33}^3 & p_{34}^3 & p_{34}^3 & p_{24}^3 & 0 \\
 p_{31}^3+p_{32}^3+p_{33}^3 & p_{33}^3 & p_{34}^3 & p_{34}^3 & p_{24}^3 & 0 \\
 p_{34}^3 & p_{34}^3 & p_{33}^3 & p_{31}^3+p_{32}^3+p_{33}^3 & p_{23}^3 & p_{13}^3\\
 p_{34}^3 & p_{34}^3 &  p_{31}^3+p_{32}^3+p_{33}^3 & p_{33}^3 & p_{23}^3 & p_{13}^3 \\
 p_{34}^2 & p_{34}^2 & p_{33}^2 & p_{33}^2 & 0 & 0\\
 0 & 0 & p_{33}^1 & p_{33}^1 & 0 & 0 
\end{pmatrix}
\]
Here, for saving space,  in each case instead of $p_{ij}^k(m-2)$ we wrote just $p_{ij}^k$.

It is now easy to compute the number of arcs in $\langle M\rangle_{\Gamma^{(m)}}$. It is
\begin{equation}\label{arcscase1}
\setlength\arraycolsep{2pt}
\begin{pmatrix}
	\bb_m(2\bb_m-1) \\ \bb_m(2\bb_m-1) \\ \bb_m(2\bb_m-1) \\ \bb_m(2\bb_m-1) \\ 2\bb_m-1 \\ 1
\end{pmatrix}^T
\begin{pmatrix}
	\bb_m(\bb_m-1) & \bb_m^2 & \bb_m(\bb_m-1 & \bb_m(\bb_m-1 & \bb_m & 0\\
	\bb_m^2 & \bb_m(\bb_m-1) & \bb_m(\bb_m-1) & \bb_m(\bb_m-1) & \bb_m & 0\\
	\bb_m(\bb_m-1) & \bb_m(\bb_m-1) & \bb_m(\bb_m-1) & \bb_m^2 & \bb_m-1 & 1\\
	\bb_m(\bb_m-1) & \bb_m(\bb_m-1) & \bb_m^2 & \bb_m(\bb_m-1) & \bb_m-1 & 1\\
	\bb_m^2 & \bb_m^2 & \bb_m(\bb_m-1) & \bb_m(\bb_m-1) & 0 & 0\\
	0 & 0 & \bb_m(2\bb_m-1) & \bb_m(2\bb_m-1) & 0 & 0
\end{pmatrix}
\begin{pmatrix}
	1 \\
	1 \\
	1 \\
	1 \\
	1 \\
	1
\end{pmatrix},
\end{equation}
where  the vector on the left hand side of this expression consists of the cardinalities of the $M_i$ ($i=1,\dots,6$).

In principle we know now the number of arcs in $\langle M\rangle_{\Gamma^{(m)}}$, but instead of computing this number outright, we stop at this point and start our consideration of  the second type of embeddings of $\Delta$ into $\Gamma^{(m)}$: 

Let $\kappa$ be an embedding of $\Delta$ into $\Gamma^{(m)} $ of the second kind. Since $\mathcal{C}^{(3)}$ is $3$-homogeneous, without loss of generality we may assume that $\bar{u}=\icol{\bar{0}\\\bar{0}}$,  $\bar{w}=\icol{\bar{0}\\\bar{e}_1}$, and $\bar{v}=\icol{\bar{e}_1\\\bar{0}}$.

\begin{lemma}
	Let $\bar{x}\in \F_2^{2m}$. Then $\bar{x}$ is a  joint neighbour of $\{\bar{u},\bar{v},\bar{w}\}$ in $\Gamma^{(m)}$ if and only if 
	\begin{enumerate}
		\item $\bar{x}_2\neq\bar{0}$,
		\item $\bar{x}_1^T\bar{x}_2 = 0$,
		\item $\bar{x}_1(1)=\bar{x}_2(1)=0$.
	\end{enumerate} 
\end{lemma}
\begin{proof}
	Clear.
\end{proof}
If we denote the  set of joint neighbours of $\{\tu,\tv,\tw\}$ in $\Gamma^{(m)}$ by $N$, then as an immediate consequence we obtain that $\langle N\rangle_{\Gamma^{(m)}}$ is isomorphic to $\Gamma_1^{(m-1)}$. So in principle we can count the arcs in $\langle N\rangle_{\Gamma^{(m)}}$. However, in order to compare this number with the data computed in the first case it is more convenient if we give a description of $\langle N\rangle_{\Gamma^{(m)}}$ with respect to a suitable equitable partition: We partition $N$ into the following $6$ parts:
\begin{align*}
	N_1 &= \{\tx\in N\mid \tx_1(m)=1, \tx_2(m)=0\},\\
	N_2 &= \{\tx\in N\mid \tx_1(m)=1, \tx_2(m)=1\},\\
	N_3 &= \{\tx\in N\mid \tx_1(m)=0, \tx_2(m)=0\},\\
	N_4 &= \{\tx\in N\mid \tx_1(m)=0, \tx_2(m)=1,(\tx_2(1),\dots,\tx_2(m-1))\neq\tn^T\},\\
	N_5 &= \{\tx\in N\mid \tx_1(m)=0, \tx_2(m)=1, \tx_1\neq\tn, (\tx_2(1),\dots,\tx_2(m-1))=\tn^T\},\\
	N_6 &= \{\tx\in N\mid \tx_1(m)=0, \tx_2(m)=1, \tx_1=\tn, (\tx_2(1),\dots,\tx_2(m-1))=\tn^T\}.  
\end{align*}
As before we examine how $\langle N\rangle_{\Gamma^{(m)}}$ looks like when projected by the projection $\Pi$ from \eqref{projection}. First we observe that the restrictions $\Pi$ to the classes $N_i$ ($i=1,\dots,6$) are all one-to-one. The projection of $\langle N\rangle_{\Gamma^{(m)}}$ with respect to the projection $\Pi$ in $\mathcal{C}^{(m)}$ is given in the following diagram:
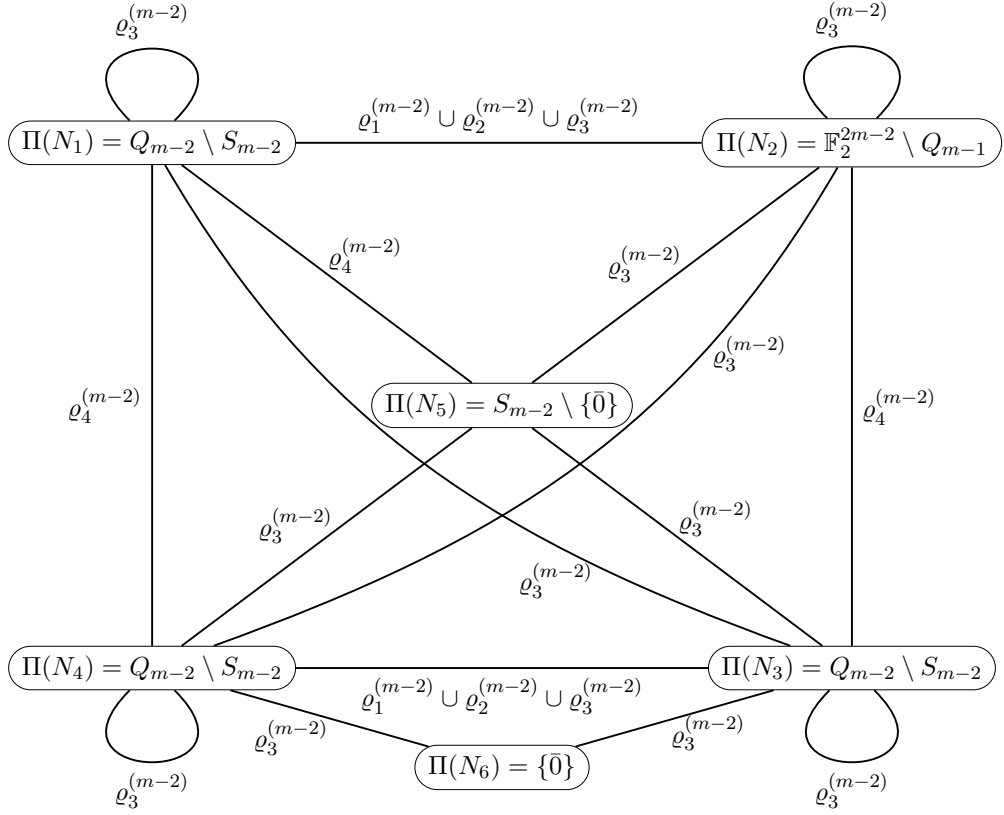
\[\scalebox{0.90}{
\begin{tikzpicture}[node distance=50mm, terminal/.style={rounded rectangle,minimum size=6mm,draw},tight/.style={inner sep=-2pt}]
\node (M5) [terminal,anchor=center] {$\Pi(N_5)=S_{m-2}\setminus\{\tn\}$};
\node (M1) [terminal, above left= of M5,anchor=center] {$\Pi(N_1)=Q_{m-2}\setminus S_{m-2}$};
\node (M2) [terminal, above right= of M5,anchor=center] {$\Pi(N_2)=\F_2^{2m-2}\setminus Q_{m-1}$};
\node (M3) [terminal, below right= of M5,anchor=center] {$\Pi(N_3)=Q_{m-2}\setminus S_{m-2}$};
\node (M4) [terminal, below left= of M5,anchor=center] {$\Pi(N_4)=Q_{m-2}\setminus S_{m-2}$};
\node (M6) [terminal, below= 50mm of M5,anchor=center] {$\Pi(N_6)=\{\tn\}$};
\path[thick, draw]  (M1) edge ["$\varrho_4^{(m-2)}$" inner sep=0pt] (M5);
\path[thick, draw]  (M1) edge [out=-60,in=160, "$\varrho_3^{(m-2)}$"  inner sep=-1pt,near end,swap](M3);
\path[thick, draw]  (M1) edge [swap, "$\varrho_4^{(m-2)}$"](M4);
\path[thick, draw]  (M1) edge ["$\varrho_1^{(m-2)}\cup\varrho_2^{(m-2)}\cup\varrho_3^{(m-2)}$"](M2);
\path[thick, draw]  (M2) edge [out=-120,in=20, "$\varrho_3^{(m-2)}$" tight, near start](M4);
\path[thick, draw]  (M2) edge [swap, "$\varrho_3^{(m-2)}$" inner sep=-3pt] (M5);
\path[thick, draw]  (M2) edge ["$\varrho_4^{(m-2)}$"](M3);
\path[thick, draw]  (M3) edge [swap,  "$\varrho_3^{(m-2)}$" inner sep=0pt] (M5);
\path[thick, draw]  (M3) edge ["$\varrho_3^{(m-2)}$" tight](M6);
\path[thick, draw]  (M3) edge ["$\varrho_1^{(m-2)}\cup\varrho_2^{(m-2)}\cup\varrho_3^{(m-2)}$"] (M4);
\path[thick, draw]  (M4) edge ["$\varrho_3^{(m-2)}$" inner sep=-3pt] (M5);
\path[thick, draw]  (M4) edge [swap,"$\varrho_3^{(m-2)}$" inner sep=0pt](M6);
\path[thick, draw] (M1) edge [loop above, min distance=20mm,in=45,out=135,"$\varrho_3^{(m-2)}$"] (M1);
\path[thick, draw] (M2) edge [loop above, min distance=20mm,in=45,out=135,"$\varrho_3^{(m-2)}$"] (M2); 
\path[thick, draw] (M3) edge [loop below, min distance=20mm,in=-45,out=-135,swap, "$\varrho_3^{(m-2)}$"] (M3); 
\path[thick, draw] (M4) edge [loop below, min distance=20mm,in=-45,out=-135,swap, "$\varrho_3^{(m-2)}$"] (M4);   
\end{tikzpicture}}
\]
Again, an edge from $\Pi(N_i)$ to $\Pi(N_j)$ labelled by $\sigma$ means that for all $\tx\in N_i$ and for all $\ty\in N_j$ we have $(\tx,\ty)\in E(\langle N\rangle_{\Gamma^{(m)}})$ if and only if $(\Pi(\tx),\Pi(\ty))\in\sigma$. From the diagram we may conclude  that $(N_1,\dots, N_6)$ is an equitable partition of $\langle N\rangle_{\Gamma^{(m)}}$. Its partition matrix is given by:
\[
\begin{pmatrix}
 p_{33}^3 & p_{41}^3+p_{42}^3+p_{43}^3 & p_{33}^3 & p_{34}^3  & p_{24}^3 & 0 \\
 p_{31}^4+p_{32}^4 + p_{33}^4 & p_{43}^4 & p_{34}^4 & p_{33}^4 &  p_{23}^4 & 0 \\
 p_{33}^3 & p_{44}^3 & p_{33}^3 & p_{31}^3+p_{32}^3+p_{33}^3 & p_{23}^3 & p_{13}^3\\
 p_{34}^3 & p_{43}^3 & p_{31}^3+p_{32}^3+p_{33}^3 & p_{33}^3 & p_{23}^3 & p_{13}^3\\
 p_{34}^2 & p_{43}^2 & p_{33}^2 & p_{33}^2 & 0 & 0 \\
 0 & 0 & p_{33}^1 & p_{33}^1 & 0 & 0 \\
\end{pmatrix}
\]
Again, for saving space,  in each case instead of $p_{ij}^k(m-2)$ we wrote just $p_{ij}^k$.

Thus, the number of arcs in $\langle N \rangle_{\Gamma^{(m)}}$ is equal to 
\[
\setlength\arraycolsep{2pt}
\begin{pmatrix}
	\bb_m(2\bb_m-1) \\ \bb_m(2\bb_m-1) \\ \bb_m(2\bb_m-1) \\ \bb_m(2\bb_m-1) \\ 2\bb_m-1 \\ 1
\end{pmatrix}^T
\begin{pmatrix}
	\bb_m(\bb_m-1) & \bb_m^2 & \bb_m(\bb_m-1) & \bb_m(\bb_m-1) & \bb_m & 0\\
	\bb_m^2 & \bb_m(\bb_m-1) & \bb_m(\bb_m-1) & \bb_m(\bb_m-1) & \bb_m & 0\\
	\bb_m(\bb_m-1) & \bb_m(\bb_m-1) & \bb_m(\bb_m-1) & \bb_m^2 & \bb_m-1 & 1\\
	\bb_m(\bb_m-1) & \bb_m(\bb_m-1) & \bb_m^2 & \bb_m(\bb_m-1) & \bb_m-1 & 1\\
	\bb_m^2 & \bb_m^2 & \bb_m(\bb_m-1) & \bb_m(\bb_m-1) & 0 & 0\\
	0 & 0 & \bb_m(2\bb_m-1) & \bb_m(2\bb_m-1) & 0 & 0 
\end{pmatrix}
\begin{pmatrix}
	1 \\
	1 \\
	1 \\
	1 \\
	1 \\
	1
\end{pmatrix},
\]
where the vector on the left hand side consists of the cardinalities of the $N_i$ ($i=1,\dots,6$). However, this is the same expression as in \eqref{arcscase1}. 
To sum up, $\#(\Gamma^{(m)},\bbT_3,\kappa)$ does not depend on the embedding $\kappa$. In other words, $\Gamma^{(m)}$ is $\bbT_3$-regular.

\subsubsection*{Proof of $\bbT_4$-regularity}
Let $\bbT_4$ be given by the following labelled diagram:
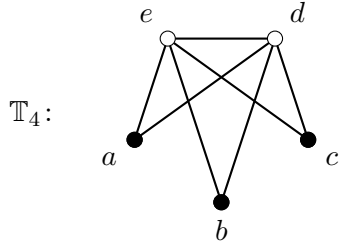
\[
	\bbT_4\colon\quad\begin{tikzpicture}[scale=0.4,baseline={(current bounding box.center)}]
		\tikzset{VertexStyle/.style = {shape = circle, draw,minimum size = 2pt, inner sep = 2pt}}        
		\SetVertexMath
		\SetVertexLabelOut
   		\Vertex[a=0*72-90,d=3,Lpos=0*72-90,Ldist=0,L=b]{B}
   		\Vertex[a=1*72-90,d=3,Lpos=1*72-90,Ldist=0,L=c]{C}
   		\Vertex[a=2*72-90,d=3,Lpos=2*72-90,Ldist=0,L=d]{D}
   		\Vertex[a=3*72-90,d=3,Lpos=3*72-90,Ldist=0,L=e]{E}
   		\Vertex[a=4*72-90,d=3,Lpos=4*72-90,Ldist=0,L=a]{A}
        \Edges(A,E,B,D,C,E,D,A)
        \AddVertexColor{black}{A,B,C}
    \end{tikzpicture}
\]
Up to symmetries of $\bbT_4$ and of $\Gamma^{(m)}$ there are $3$ types of embeddings of $\Delta=\overline{K}_3$ into $\Gamma^{(m)}$: Fix an embedding $\kappa\colon\Delta\injto\Gamma^{(m)}$. If we assume that  $\kappa\colon a\mapsto\tu, b\mapsto\tv, c\mapsto\tw$, then the subcolorgraph of $\mathcal{C}^{(m)}$ induced by $\{\tu,\tv,\tw\}$ is one of the following:
	\begin{align*}
		(1) &:	
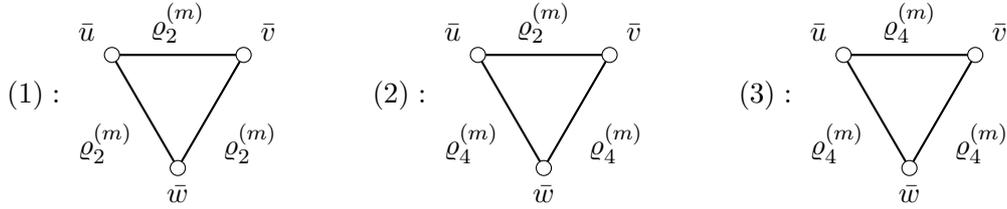
		\begin{tikzpicture}[baseline={(current bounding box.center)}]
		\tikzset{VertexStyle/.style = {shape = circle, draw,minimum size = 2pt, inner sep = 2pt}}        
		\SetVertexLabelOut
	    \SetVertexMath
	    \SetVertexLabel
   		\Vertex[a=0*120-90,d=1,Lpos=0*120-90,Ldist=0,L=\bar{w}]{c}
   		\Vertex[a=1*120-90,d=1,Lpos=1*120-90,Ldist=0,L=\bar{v}]{a}
   		\Vertex[a=2*120-90,d=1,Lpos=2*120-90,Ldist=0,L=\bar{u}]{z}
		\Edge[labelstyle={MyLabel,swap}, label=$\varrho_2^{(m)}$](z)(a) 
		\Edge[labelstyle={MyLabel}, label={$\varrho_2^{(m)}$}](z)(c) 
		\Edge[labelstyle={MyLabel,swap}, label={$\varrho_2^{(m)}$}](a)(c) 
	   \end{tikzpicture} 
	   & (2) &:
   		\begin{tikzpicture}[baseline={(current bounding box.center)}]
		\tikzset{VertexStyle/.style = {shape = circle, draw,minimum size = 2pt, inner sep = 2pt}}        
		\SetVertexLabelOut
	    \SetVertexMath
	    \SetVertexLabel
   		\Vertex[a=0*120-90,d=1,Lpos=0*120-90,Ldist=0,L=\bar{w}]{c}
   		\Vertex[a=1*120-90,d=1,Lpos=1*120-90,Ldist=0,L=\bar{v}]{a}
   		\Vertex[a=2*120-90,d=1,Lpos=2*120-90,Ldist=0,L=\bar{u}]{z}
		\Edge[labelstyle={MyLabel,swap}, label=$\varrho_2^{(m)}$](z)(a) 
		\Edge[labelstyle={MyLabel}, label={$\varrho_4^{(m)}$}](z)(c) 
		\Edge[labelstyle={MyLabel,swap}, label={$\varrho_4^{(m)}$}](a)(c) 
	   \end{tikzpicture} 
	   & (3) &:
   		\begin{tikzpicture}[baseline={(current bounding box.center)}]
		\tikzset{VertexStyle/.style = {shape = circle, draw,minimum size = 2pt, inner sep = 2pt}}        
		\SetVertexLabelOut
	    \SetVertexMath
	    \SetVertexLabel
   		\Vertex[a=0*120-90,d=1,Lpos=0*120-90,Ldist=0,L=\bar{w}]{c}
   		\Vertex[a=1*120-90,d=1,Lpos=1*120-90,Ldist=0,L=\bar{v}]{a}
   		\Vertex[a=2*120-90,d=1,Lpos=2*120-90,Ldist=0,L=\bar{u}]{z}
		\Edge[labelstyle={MyLabel,swap}, label=$\varrho_4^{(m)}$](z)(a) 
		\Edge[labelstyle={MyLabel}, label={$\varrho_4^{(m)}$}](z)(c) 
		\Edge[labelstyle={MyLabel,swap}, label={$\varrho_4^{(m)}$}](a)(c) 
	   \end{tikzpicture} 
	\end{align*}
\paragraph{About the first type of embeddings:} Since $\mathcal{C}^{(m)}$ is $3$-homogeneous, without loss of generality we can assume that $\tu=\icol{\tn\\\tn}$, $\tv=\icol{\te_1\\\tn}$, $\tw=\icol{\te_2\\\tn}$. 
\begin{lemma}
	Let $\bar{x}\in \F_2^{2m}$. Then $\bar{x}$ is a  joint neighbour of $\{\bar{u},\bar{v},\bar{w}\}$ in $\Gamma^{(m)}$ if and only if 
	\begin{enumerate}
		\item $\bar{x}_2\neq\bar{0}$,
		\item $\bar{x}_1^T\bar{x}_2 = 0$,
		\item $\bar{x}_2(1)=\bar{x}_2(2)=0$.
	\end{enumerate} 
\end{lemma}
\begin{proof}
	Clear.
\end{proof}
As it was done before, the set $M$ of joint neighbors of $\{\tu,\tv,\tw\}$ in $\Gamma^{(m)}$ is subdivided into subsets:
\begin{align*}
	M_1 &= \{\tx\in M\mid \tx_1(1)=0, \tx_1(2)=0, \tx_2(1)=0, \tx_2(2)=0\},\\
	M_2 &= \{\tx\in M\mid \tx_1(1)=0, \tx_1(2)=1, \tx_2(1)=0, \tx_2(2)=0\},\\
	M_3 &= \{\tx\in M\mid \tx_1(1)=1, \tx_1(2)=0, \tx_2(1)=0, \tx_2(2)=0\},\\
	M_4 &= \{\tx\in M\mid \tx_1(1)=1, \tx_1(2)=1, \tx_2(1)=0, \tx_2(2)=0\}.
\end{align*}
Consider the projection
\begin{align} \label{projction2}
	\Pi\colon\F_2^{2m}\epito\F_2^{2m-4},\quad 
	\begin{pmatrix}
		\tx_1\\\tx_2
	\end{pmatrix}\mapsto
	\begin{pmatrix}
		\tilde{x}_1\\\tilde{x}_2
	\end{pmatrix},
\end{align}
where $\tilde{x}_i$ is the unique vector from $\F_2^{m-2}$, such that $\tx_i=\icol{\tx_i(1)\\\tx_i(2)\\\tilde{x}_i}$ ($i=1,2$). Note that for each $i\in\{1,\dots,4\}$ the restriction $\Pi\restr_{M_i}$ is one-to-one. The projection of $\langle M \rangle_{\Gamma^{(m)}}$ under $\Pi$ looks as follows:
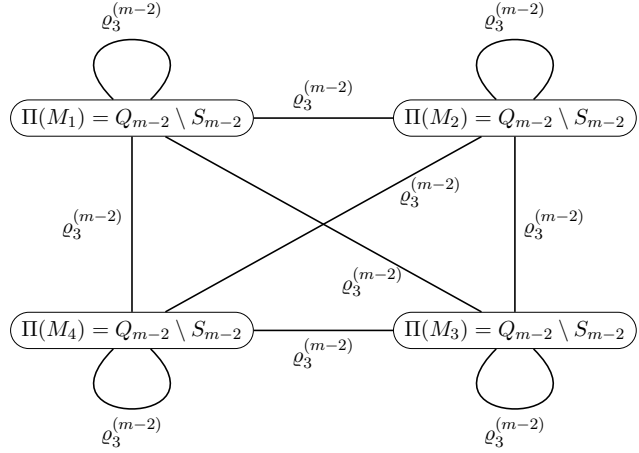
\[\scalebox{0.75}{
\begin{tikzpicture}[node distance=20mm, terminal/.style={rounded rectangle,minimum size=6mm,draw},tight/.style={inner sep=-2pt}]
\node (M5) [anchor=center]{};
\node (M1) [terminal, above left= of M5,anchor=south east] {$\Pi(M_1)=Q_{m-2}\setminus S_{m-2}$};
\node (M2) [terminal, above right= of M5,anchor=south west] {$\Pi(M_2)=Q_{m-2}\setminus S_{m-2}$};
\node (M3) [terminal, below right= of M5,anchor=north west] {$\Pi(M_3)=Q_{m-2}\setminus S_{m-2}$};
\node (M4) [terminal, below left= of M5,anchor=north east] {$\Pi(M_4)=Q_{m-2}\setminus S_{m-2}$};
\path[thick, draw]  (M1) edge ["$\varrho_3^{(m-2)}$"  inner sep=-1pt,near end,swap](M3);
\path[thick, draw]  (M1) edge [swap, "$\varrho_3^{(m-2)}$"](M4);
\path[thick, draw]  (M1) edge ["$\varrho_3^{(m-2)}$"](M2);
\path[thick, draw]  (M2) edge ["$\varrho_3^{(m-2)}$" tight, near start](M4);
\path[thick, draw]  (M2) edge ["$\varrho_3^{(m-2)}$"](M3);
\path[thick, draw]  (M3) edge ["$\varrho_3^{(m-2)}$"] (M4);
\path[thick, draw] (M1) edge [loop above, min distance=20mm,in=45,out=135,"$\varrho_3^{(m-2)}$"] (M1);
\path[thick, draw] (M2) edge [loop above, min distance=20mm,in=45,out=135,"$\varrho_3^{(m-2)}$"] (M2); 
\path[thick, draw] (M3) edge [loop below, min distance=20mm,in=-45,out=-135,swap, "$\varrho_3^{(m-2)}$"] (M3); 
\path[thick, draw] (M4) edge [loop below, min distance=20mm,in=-45,out=-135,swap, "$\varrho_3^{(m-2)}$"] (M4);   
\end{tikzpicture}}
\]
Thus, $(M_1,M_2,M_3,M_4)$ is an equitable partition of $\langle M \rangle_{\Gamma^{(m)}}$. Its partition matrix is given by:
\[
\begin{pmatrix}
 p_{33}^3(m-2) & p_{33}^3(m-2) & p_{33}^3(m-2) & p_{33}^3(m-2)\\
 p_{33}^3(m-2) & p_{33}^3(m-2) & p_{33}^3(m-2) & p_{33}^3(m-2)\\
 p_{33}^3(m-2) & p_{33}^3(m-2) & p_{33}^3(m-2) & p_{33}^3(m-2)\\
 p_{33}^3(m-2) & p_{33}^3(m-2) & p_{33}^3(m-2) & p_{33}^3(m-2)
 \end{pmatrix}
\]
Thus, the number of arcs in $\langle M \rangle_{\Gamma^{(m)}}$ is given by:
\[
\begin{pmatrix}
	\bb_m(2\bb_m-1)\\
	\bb_m(2\bb_m-1)\\
	\bb_m(2\bb_m-1)\\
	\bb_m(2\bb_m-1)	
\end{pmatrix}^T
\begin{pmatrix}
	\bb_m(\bb_m-1) & \bb_m(\bb_m-1) & \bb_m(\bb_m-1) & \bb_m(\bb_m-1)\\ 
	\bb_m(\bb_m-1) & \bb_m(\bb_m-1) & \bb_m(\bb_m-1) & \bb_m(\bb_m-1)\\ 
	\bb_m(\bb_m-1) & \bb_m(\bb_m-1) & \bb_m(\bb_m-1) & \bb_m(\bb_m-1)\\ 
	\bb_m(\bb_m-1) & \bb_m(\bb_m-1) & \bb_m(\bb_m-1) & \bb_m(\bb_m-1)
\end{pmatrix}
\begin{pmatrix}
	1\\
	1\\
	1\\
	1
\end{pmatrix}.
\]

\paragraph{About the second type of embeddings:} Since $\mathcal{C}^{(m)}$ is $3$-homogeneous, without loss of generality we can assume that $\tu=\icol{\tn\\\tn}$, $\tv=\icol{\te_1\\\tn}$, $\tw=\icol{\te_2\\\te_2}$.
\begin{lemma}
	Let $\tx\in \F_2^{2m}$. Then $\tx$ is a  joint neighbour of $\{\tu,\tv,\tw\}$ in $\Gamma^{(m)}$ if and only if 
	\begin{enumerate}
		\item $\tx_2\neq\tn$,
		\item $\tx_1^T\tx_2 = 0$,
		\item $\tx_2(1)=0$,
		\item $\tx_1(2)\neq\tx_2(2)$,
		\item $\tx_2(2)=1\Longrightarrow(\tx_2(3),\dots,\tx_2(m))\neq\tn^T$.
	\end{enumerate} 
\end{lemma}
\begin{proof}
	Clear.
\end{proof}
As usually, we decompose the set $M$ of joint neighbours of $\{\bar{u},\bar{v},\bar{w}\}$ in $\Gamma^{(m)}$:
\begin{align*}
	M_1 &= \{\tx\in M\mid \tx_1(1)=0, \tx_1(2)=0, \tx_2(1)=0, \tx_2(2)=1\},\\
	M_2 &= \{\tx\in M\mid \tx_1(1)=0, \tx_1(2)=1, \tx_2(1)=0, \tx_2(2)=0\},\\	
	M_3 &= \{\tx\in M\mid \tx_1(1)=1, \tx_1(2)=0, \tx_2(1)=0, \tx_2(2)=1\},\\	
	M_4 &= \{\tx\in M\mid \tx_1(1)=1, \tx_1(2)=1, \tx_2(1)=0, \tx_2(2)=0\},	
\end{align*}
and show that $(M_1,M_2,M_3,M_4)$ forms an equitable partition of $\langle M\rangle_{\Gamma^{(m)}}$. For each $i\in\{1,\dots,4\}$, the restriction of the projection $\Pi$ from \eqref{projction2} to $M_i$ is one-to-one. The projection of $\langle M \rangle_{\Gamma^{(m)}}$ under $\Pi$ in $\mathcal{C}^{(m-2)}$ is given by 
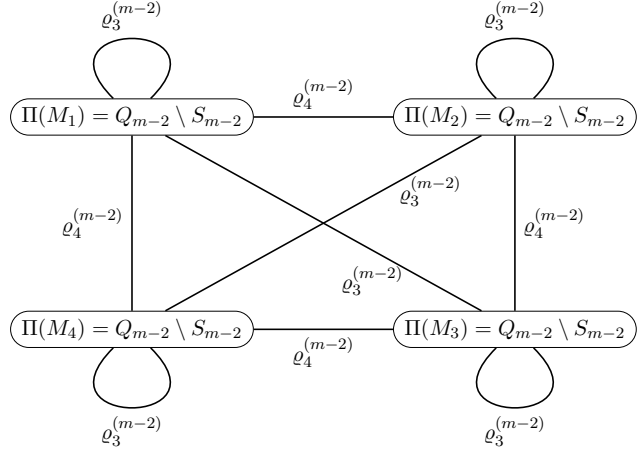
\[\scalebox{0.75}{
\begin{tikzpicture}[node distance=20mm, terminal/.style={rounded rectangle,minimum size=6mm,draw},tight/.style={inner sep=-2pt}]
\node(M5)[anchor=center]{};
\node (M1) [terminal, above left= of M5,anchor=south east] {$\Pi(M_1)=Q_{m-2}\setminus S_{m-2}$};
\node (M2) [terminal, above right= of M5,anchor=south west] {$\Pi(M_2)=Q_{m-2}\setminus S_{m-2}$};
\node (M3) [terminal, below right= of M5,anchor=north west] {$\Pi(M_3)=Q_{m-2}\setminus S_{m-2}$};
\node (M4) [terminal, below left= of M5,anchor=north east] {$\Pi(M_4)=Q_{m-2}\setminus S_{m-2}$};
\path[thick, draw]  (M1) edge ["$\varrho_3^{(m-2)}$"  inner sep=-1pt,near end,swap](M3);
\path[thick, draw]  (M1) edge [swap, "$\varrho_4^{(m-2)}$"](M4);
\path[thick, draw]  (M1) edge ["$\varrho_4^{(m-2)}$"](M2);
\path[thick, draw]  (M2) edge ["$\varrho_3^{(m-2)}$" tight, near start](M4);
\path[thick, draw]  (M2) edge ["$\varrho_4^{(m-2)}$"](M3);
\path[thick, draw]  (M3) edge ["$\varrho_4^{(m-2)}$"] (M4);
\path[thick, draw] (M1) edge [loop above, min distance=20mm,in=45,out=135,"$\varrho_3^{(m-2)}$"] (M1);
\path[thick, draw] (M2) edge [loop above, min distance=20mm,in=45,out=135,"$\varrho_3^{(m-2)}$"] (M2); 
\path[thick, draw] (M3) edge [loop below, min distance=20mm,in=-45,out=-135,swap, "$\varrho_3^{(m-2)}$"] (M3); 
\path[thick, draw] (M4) edge [loop below, min distance=20mm,in=-45,out=-135,swap, "$\varrho_3^{(m-2)}$"] (M4);   
\end{tikzpicture}}
\]
Consequently, $(M_1,\dots,M_4)$ is an equitable partition of $\langle M \rangle_{\Gamma^{(m)}}$. Its partition matrix is given by:
\[
\begin{pmatrix}
p_{33}^3(m-2) & p_{34}^3(m-2) & p_{33}^3(m-2) & p_{34}^3(m-2)\\
p_{34}^3(m-2) & p_{33}^3(m-2) & p_{34}^3(m-2) & p_{33}^3(m-2)\\
p_{33}^3(m-2) & p_{34}^3(m-2) & p_{33}^3(m-2) & p_{34}^3(m-2)\\
p_{34}^3(m-2) & p_{33}^3(m-2) & p_{34}^3(m-2) & p_{33}^3(m-2)
 \end{pmatrix}
\]
Thus, the number of arcs in $\langle M \rangle_{\Gamma^{(m)}}$ is equal to 
\[
\begin{pmatrix}
	\bb_m(2\bb_m-1)\\
	\bb_m(2\bb_m-1)\\
	\bb_m(2\bb_m-1)\\
	\bb_m(2\bb_m-1)	
\end{pmatrix}^T
\begin{pmatrix}
	\bb_m(\bb_m-1) & \bb_m(\bb_m-1) & \bb_m(\bb_m-1) & \bb_m(\bb_m-1)\\ 
	\bb_m(\bb_m-1) & \bb_m(\bb_m-1) & \bb_m(\bb_m-1) & \bb_m(\bb_m-1)\\ 
	\bb_m(\bb_m-1) & \bb_m(\bb_m-1) & \bb_m(\bb_m-1) & \bb_m(\bb_m-1)\\ 
	\bb_m(\bb_m-1) & \bb_m(\bb_m-1) & \bb_m(\bb_m-1) & \bb_m(\bb_m-1)
\end{pmatrix}
\begin{pmatrix}
	1\\
	1\\
	1\\
	1
\end{pmatrix}.
\]

\paragraph{About the third type of embeddings:} Since $\mathcal{C}^{(m)}$ is $3$-homogeneous, without loss of generality we can assume that $\tu=\icol{\tn\\\tn}$, $\tv=\icol{\te_1\\\te_1}$, $\tw=\icol{\te_1+\te_2\\\te_2}$.
\begin{lemma}
	Let $\tx\in \F_2^{2m}$. Then $\tx$ is a  joint neighbour of $\{\tu,\tv,\tw\}$ in $\Gamma^{(m)}$ if and only if 
	\begin{enumerate}
		\item $\tx_2\neq\tn$,
		\item $\tx_1^T\tx_2 = 0$,
		\item $\tx_1(1)\neq\tx_2(1)$,
		\item $\tx_2(2)=\tx_1(2)+\tx_1(1)$,
		\item $\tx_2(1)=1\Longrightarrow (\tx_2(2),\dots,\tx_2(m))\neq\tn^T$,
		\item $\tx_2(2)=1\Longrightarrow (\tx_2(1),\tx_2(3),\dots,\tx_2(m))\neq\tn^T$,
		\item $\tx_2(1)=\tx_2(2)=1\Longrightarrow(\tx_2(3),\dots,\tx_2(m))\neq\tn^T$.
	\end{enumerate} 
\end{lemma}
\begin{proof}
	Clear.
\end{proof}
As usually, we decompose the set $M$ of joint neighbours of $\{\bar{u},\bar{v},\bar{w}\}$ in $\Gamma^{(m)}$:
\begin{align*}
	M_1 &= \{\tx\in M\mid \tx_1(1)=0, \tx_1(2)=0, \tx_2(1)=1, \tx_2(2)=0\},\\
	M_2 &= \{\tx\in M\mid \tx_1(1)=0, \tx_1(2)=1, \tx_2(1)=1, \tx_2(2)=1\},\\	
	M_3 &= \{\tx\in M\mid \tx_1(1)=1, \tx_1(2)=0, \tx_2(1)=0, \tx_2(2)=1\},\\	
	M_4 &= \{\tx\in M\mid \tx_1(1)=1, \tx_1(2)=1, \tx_2(1)=0, \tx_2(2)=0\}.	
\end{align*}
For each $i\in\{1,\dots,4\}$, the restriction of the projection $\Pi$ from \eqref{projction2} to $M_i$ is one-to-one. The projection of $\langle M\rangle_{\Gamma^{(m)}}$ under $\Pi$ in $\mathcal{C}^{(m-2)}$ is given by 
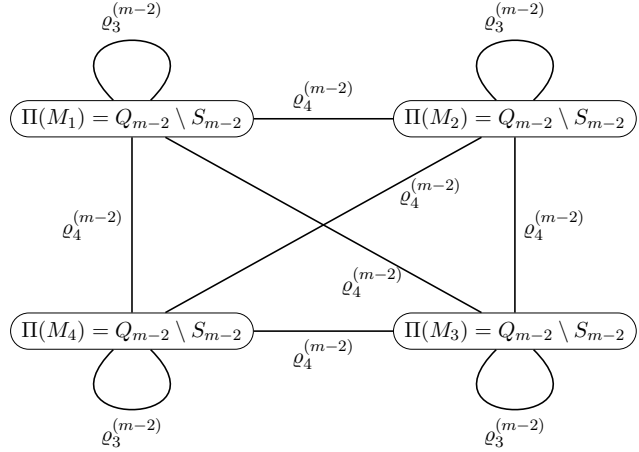
\[\scalebox{0.75}{
\begin{tikzpicture}[node distance=20mm, terminal/.style={rounded rectangle,minimum size=6mm,draw},tight/.style={inner sep=-2pt}]
\node (M5) [anchor=center]{};
\node (M1) [terminal, above left= of M5,anchor=south east] {$\Pi(M_1)=Q_{m-2}\setminus S_{m-2}$};
\node (M2) [terminal, above right= of M5,anchor=south west] {$\Pi(M_2)=Q_{m-2}\setminus S_{m-2}$};
\node (M3) [terminal, below right= of M5,anchor=north west] {$\Pi(M_3)=Q_{m-2}\setminus S_{m-2}$};
\node (M4) [terminal, below left= of M5,anchor=north east] {$\Pi(M_4)=Q_{m-2}\setminus S_{m-2}$};
\path[thick, draw]  (M1) edge ["$\varrho_4^{(m-2)}$"  inner sep=-1pt,near end,swap](M3);
\path[thick, draw]  (M1) edge [swap, "$\varrho_4^{(m-2)}$"](M4);
\path[thick, draw]  (M1) edge ["$\varrho_4^{(m-2)}$"](M2);
\path[thick, draw]  (M2) edge ["$\varrho_4^{(m-2)}$" tight, near start](M4);
\path[thick, draw]  (M2) edge ["$\varrho_4^{(m-2)}$"](M3);
\path[thick, draw]  (M3) edge ["$\varrho_4^{(m-2)}$"] (M4);
\path[thick, draw] (M1) edge [loop above, min distance=20mm,in=45,out=135,"$\varrho_3^{(m-2)}$"] (M1);
\path[thick, draw] (M2) edge [loop above, min distance=20mm,in=45,out=135,"$\varrho_3^{(m-2)}$"] (M2); 
\path[thick, draw] (M3) edge [loop below, min distance=20mm,in=-45,out=-135,swap, "$\varrho_3^{(m-2)}$"] (M3); 
\path[thick, draw] (M4) edge [loop below, min distance=20mm,in=-45,out=-135,swap, "$\varrho_3^{(m-2)}$"] (M4);   
\end{tikzpicture}} 
\]
Consequently, $(M_1,M_2,M_3,M_4)$ forms an equitable partition of $\langle M \rangle_{\Gamma^{(m)}}$. Its partition matrix is given by:
\[
\begin{pmatrix}
 p_{33}^3(m-2) & p_{34}^3(m-2) & p_{34}^3(m-2) & p_{34}^3(m-2)\\
 p_{34}^3(m-2) & p_{33}^3(m-2) & p_{34}^3(m-2) & p_{34}^3(m-2)\\
 p_{34}^3(m-2) & p_{34}^3(m-2) & p_{33}^3(m-2) & p_{34}^3(m-2)\\
 p_{34}^3(m-2) & p_{34}^3(m-2) & p_{34}^3(m-2) & p_{33}^3(m-2)
 \end{pmatrix}
\]
Thus, the number of arcs in $\langle M \rangle_{\Gamma^{(m)}}$ is equal to 
\[
\begin{pmatrix}
	\bb_m(2\bb_m-1)\\
	\bb_m(2\bb_m-1)\\
	\bb_m(2\bb_m-1)\\
	\bb_m(2\bb_m-1)	
\end{pmatrix}^T
\begin{pmatrix}
	\bb_m(\bb_m-1) & \bb_m(\bb_m-1) & \bb_m(\bb_m-1) & \bb_m(\bb_m-1)\\ 
	\bb_m(\bb_m-1) & \bb_m(\bb_m-1) & \bb_m(\bb_m-1) & \bb_m(\bb_m-1)\\ 
	\bb_m(\bb_m-1) & \bb_m(\bb_m-1) & \bb_m(\bb_m-1) & \bb_m(\bb_m-1)\\ 
	\bb_m(\bb_m-1) & \bb_m(\bb_m-1) & \bb_m(\bb_m-1) & \bb_m(\bb_m-1)
\end{pmatrix}
\begin{pmatrix}
	1\\
	1\\
	1\\
	1
\end{pmatrix}.
\]

Note that in all three cases we counted the same number of arcs in $\langle M \rangle_{\Gamma^{(m)}}$. Thus, the number $\#(\Gamma^{(m)},\bbT_4,\kappa)$ does not depend on $\kappa$. In other words, $\Gamma^{(m)}$ is $\bbT_4$-regular. 

This finishes the proof that $\Gamma^{(m)}$ is $(3,5)$-regular.\qed

\section{Outlook}
In \cite{Iva94}, A.V.~Ivanov described another series of $(2,4)$-regular graphs whose existence is related to the unique (up to equivalence) non-degenerate quadratic form of Witt-index $m-1$ on $\F_2^{2m}$. Let $\hq^{(m)}$ be such a  form.
 By $\hQ_m$ we denote the quadric defined by $\hq^{(m)}$, and by $\hS_m$ a maximal singular subspace.  The bilinear form associated with $\hq^{(m)}$ is given by $[\tx,\ty]^{(m)}=\hq^{(m)}(\tx+\ty)+\hq^{(m)}(\tx)+\hq^{(m)}(\ty)$. The given data give rise to the following five binary relations on $\F_2^{2m}$:
\begin{align*}
	\sigma_1^{(m)} &= \{(\tv,\tw)\mid \tv=\tw\}, & \sigma_2^{(m)} &=\{(\tv,\tw)\mid \tv+\tw\in\hS_m\setminus\{\tn\}\},\\
	\sigma_3^{(m)} &=\{(\tv,\tw)\mid\tv+\tw\in\hQ_m\setminus\hS_m\}, & \sigma_4^{(m)} &=\{(\tv,\tw)\mid \tv+\tw\in\hS_m^\perp\setminus\hS_m\},\\
	\sigma_5^{(m)} &=\{(\tv,\tw)\mid \tv+\tw\in \F_2^{2m}\setminus(\hQ_m\cup\hS_m^\perp)\}.
\end{align*}
It is known (cf.~ \cite{Iva94})  that the relational structure  $\hcC^{(m)}\coloneqq (\F_2^{2m},\sigma_1^{(m)},\sigma_2^{(m)},\sigma_3^{(m)},\sigma_4^{(m)},\sigma_5^{(m)})$ is a coherent configuration and that the graph $\hGamma^{(m)}\coloneqq (\F_2^{2m},\sigma_2^{(m)}\cup\sigma_5^{(m)})$ is $(2,4)$-regular. In the course of our research of the Brouwer-Ivanov-Klin-graphs we had also a look onto the graphs $\hGamma^{(m)}$. So far we were able to show that for all $m\ge 5$ we have that: 
	\begin{enumerate}
		\item $\Aut(\hGamma^{(m)})=\Aut(\hcC^{(m)})$,
		\item $\hcC^{(m)}$ is $3$-homogeneous,
		\item $\hGamma^{(m)}$ is $(3,5)$-regular,
		\item $\hGamma^{(m)}$ is not $2$-homogeneous.
	\end{enumerate} 
The proof is postponed to a subsequent publication, as it uses different techniques and would explode the size of this paper.

%\bibliographystyle{abbrv} 
%\bibliography{PP11}

\end{document}